\documentclass[10pt]{amsart}

\usepackage{extarrows}
\usepackage{amsmath,arydshln,multirow}
\usepackage[alphabetic]{amsrefs}
\usepackage{arydshln}
\usepackage{cases}
\usepackage{amsmath}
\usepackage{amsfonts}
\usepackage{bm}
\usepackage{arydshln}
\usepackage{amsfonts,amsmath,amssymb,amscd,bbm,amsthm,mathrsfs,dsfont}
\usepackage{mathrsfs}
\usepackage{pb-diagram}
\usepackage{amssymb}
\usepackage[all]{xy}
\usepackage[dvips]{graphicx}
\usepackage[dvipsnames]{xcolor}
\usepackage{mathtools}
\usepackage{tikz}
\usepackage[top=25mm,bottom=25mm,left=30mm,right=30mm]{geometry}
\newtheorem{theorem}{Theorem}[section]

\newtheorem{proposition}[theorem]{Proposition}
\newtheorem{lemma}[theorem]{Lemma}

\theoremstyle{definition}
\newtheorem{definition}[theorem]{Definition}
\newtheorem{example}[theorem]{Example}
\newtheorem{proposition-definition}[theorem]{Proposition-Definition}
\newtheorem{corollary}[theorem]{Corollary}

\theoremstyle{remark}
\newtheorem{remark}[theorem]{Remark}

\numberwithin{equation}{section}

\def\ee{\mathbf{e}}

\def\gg{\mathbf{g}}

\def\xx{\mathbf{x}}
\def\yy{\mathbf{y}}

\def\TT{\mathbb{T}}

\def\PP{\mathbb{P}}
\def\ZZ{\mathbb{Z}}

\def\Acal{\mathcal{A}}

\def\Fcal{\mathcal{F}}

\def\trop{\mathrm{Trop}}
\def\mod{\opname{mod}\nolimits}

\newcommand{\opname}[1]{\operatorname{\mathsf{#1}}}

\newcommand{\Hom}{\opname{Hom}}

\newcommand{\Fac}{\opname{Fac}}
\newcommand{\add}{\opname{add}\nolimits}

\begin{document}

\title{Bongartz completion via $c$-vectors}


\author{Peigen Cao}
\address{Einstein Institute of Mathematics, Edmond J. Safra Campus, The Hebrew University of Jerusalem, Jerusalem
91904, Israel}
\email{peigencao@126.com}

\author{Yasuaki Gyoda}
\address{Graduate School of Mathematical Sciences, The University of Tokyo, 3-8-1 Komaba Meguro-ku Tokyo 153-8914, Japan
}
\email{gyoda-yasuaki@g.ecc.u-tokyo.ac.jp}

\author{Toshiya Yurikusa}
\address{Mathematical Institute, Tohoku University, Aoba-ku, Sendai, 980-8578, Japan}
\email{toshiya.yurikusa.d8@tohoku.ac.jp}

\subjclass[2010]{13F60, 16G20}

\date{}

\dedicatory{}

\keywords{cluster algebra, $\tau$-tilting theory, Bongartz completion, $c$-vectors, exchange quiver}

\begin{abstract}
In the present paper, we first give a characterization for Bongartz completion in $\tau$-tilting theory via $c$-vectors. Motivated by this characterization, we give the definition of Bongartz completion in cluster algebras using $c$-vectors. Then we prove the existence and uniqueness of Bongartz completion in cluster algebras. We also prove that Bongartz completion admits certain commutativity. We give two applications for Bongartz completion in cluster algebras. As the first application, we prove the full subquiver of the exchange quiver (or known as oriented exchange graph) of a cluster algebra $\mathcal A$ whose vertices consist of the seeds of $\mathcal A$ containing particular cluster variables is isomorphic to the exchange quiver of another cluster algebra. As the second application, we prove that in a $Y$-pattern over a universal semifield, each $Y$-seed (up to a $Y$-seed equivalence) is uniquely determined by the negative $y$-variables in this $Y$-seed.
\end{abstract}

\maketitle


\tableofcontents

\section{Introduction}
Cluster algebras, introduced by Fomin and Zelevinsky \cite{fzi}
around the year 2000, are commutative algebras whose generators
and relations are constructed in a recursive manner.
Among these algebras, there are the algebras of homogeneous coordinates on the Grassmannians,
on the flag varieties and on many other varieties which play
an important role in geometry and representation theory.

Two kinds of integer matrices called $C$-matrices and $G$-matrices whose columns are known as $c$-vectors and $g$-vectors, respectively, are introduced in \cite{fziv} and have played important roles in the theory of cluster algebras. They have the so-called sign-coherence property (see Theorem \ref{thm:signs-ci}), which was conjectured in \cite{fziv} and was proved in \cite{GHKK} in the general case.

Adachi, Iyama and Reiten \cite{air} introduced $\tau$-tilting theory to complete classical tilting theory from the viewpoint of mutations. In the context of $\tau$-tilting theory, there also exists the notion of $g$-vectors \cite{air} and $c$-vectors \cite{fu2017}, which are parallel to the corresponding concepts in cluster algebras.

In classic tilting theory, Bongartz completion \cite{bong} is an operation to complete a partial tilting module to a tilting module. Bongartz completion is naturally extended to $\tau$-tilting theory, which completes a $\tau$-rigid pair to a $\tau$-tilting pair \cite{air}. In cluster algebras, the first-named author and Li \cite{cl2} proved that for any cluster ${\bf x}_t$ and any subset $U$ of some cluster ${\bf x}_{t_0}$, there is a canonical method using $g$-vectors to complete $U$ to a cluster ${\bf x}_{t^\prime}$. This is known as the enough $g$-pairs property. In \cite{cao}, it was found that the completion in cluster algebras using $g$-vectors is actually Bongartz co-completion, that is, the dual of Bongartz completion, if one looks at it in a Hom-finite $2$-Calabi-Yau triangulated category with cluster tilting objects \cite{bmrrt}.

In the present paper, we first give a characterization for Bongartz completion in $\tau$-tilting theory via $c$-vectors (see Theorem \ref{thm-tau}). Motivated by this characterization, we give the definition of Bongartz completion in cluster algebras using $c$-vectors. Then we prove the existence and uniqueness of Bongartz completion in cluster algebras using the enough $g$-pairs property and the tropical dualities between $C$-matrices and $G$-matrices \cite{nz} (see Theorem \ref{bongartz-completion}). We also prove that Bongartz completion admits certain commutativity (see Theorem \ref{thmcommu}).

 Finally, we give two applications of Bongartz completion in cluster algebras. As the first application, we prove the full subquiver of the exchange quiver (or known as oriented exchange graph) of a cluster algebra $\mathcal A$ whose vertices consist of the seeds of $\mathcal A$ containing particular cluster variables is isomorphic to the exchange quiver of another cluster algebra (see Theorem \ref{thmquiver}). As the second application, we prove that in a $Y$-pattern over a universal semifield, each $Y$-seed (up to a $Y$-seed equivalence) is uniquely determined by the negative $y$-variables in this $Y$-seed (see Theorem \ref{thmpoisson}).

\section{$\tau$-tilting theory}

In this section, we recall $\tau$-tilting theory \cite{air}. We fix a finite dimensional basic algebra $A$ over an algebraically closed field $K$. We denote by $\mod A$ the category of finitely generated left $A$-modules and by $\tau$ the Auslander-Reiten translation in $\mod A$.

For a full subcategory $\mathcal C$ of $\mod A$, we define some full subcategories of $\mod A$ as follows:
\begin{itemize}
\item $\add\mathcal C$ is the additive closure of $\mathcal C$;
\item $\Fac\mathcal C$ consists of the factor modules of objects in $\add\mathcal C$;
\item $\mathcal C^\bot$ and $\prescript{\bot}{}{\mathcal C}$ are the full subcategories given by
\begin{eqnarray}
 \mathcal C^\bot&:=&\{X\in \mod A| \Hom_{A}(\mathcal C, X)=0 \},\nonumber\\
  \prescript{\bot}{}{\mathcal C}&:=&\{X\in \mod A | \Hom_{A}(X,\mathcal C)=0 \}.\nonumber
\end{eqnarray}
\end{itemize}

\subsection{$\tau$-tilting pair}
Given a module $M\in\mod A$,
\begin{itemize}
    \item we say that $M$ is a \emph{brick} if $\Hom_A(M,M)$ is a division ring;
    \item we say that $M$ is  \emph{$\tau$-rigid} if $\Hom_A(M,\tau M)=0$;
    \item we denote by $|M|$ the number of non-isomorphic indecomposable direct summands of $M$.
\end{itemize}

\begin{definition}[$\tau$-rigid and $\tau$-tilting pair]
Let $M$ be a module in $\mod A$ and $P$ a projective module in $\mod A$. The pair $(M,P)$ is called \emph{$\tau$-rigid} if $M$ is $\tau$-rigid and $\Hom_A(P,M)=0$.
It is called \emph{$\tau$-tilting} (resp., \emph{almost $\tau$-tilting}) if, moreover, it satisfies $|M|+|P|=n=|A|$ (resp., $|M|+|P|=n-1=|A|-1$).
\end{definition}

Note that a module $M\in \mathrm{mod} A$ is called a \emph{support $\tau$-tilting module} if there exists a basic projective module $P$ such that $(M,P)$ is a $\tau$-tilting pair. In fact, such a basic projective module $P$ is unique up to isomorphisms \cite{air}*{Proposition 2.3}.
A $\tau$-rigid pair $(M,P)$ is \emph{indecomposable} if $M\oplus P$ is indecomposable in $\mod A$ and it is \emph{basic} if both $M$ and $P$ are basic in $\mod A$. In the sequel, we always assume that the $\tau$-rigid pairs are basic and consider them up to isomorphisms.

\begin{theorem}[\cite{air}*{Theorem 2.18}] \label{thmair}
Any basic almost $\tau$-tilting pair $(U,Q)$ in $\mod A$ is a direct summand of exactly two basic $\tau$-tilting pairs $(M,P)$ and $(M^\prime,P^\prime)$ in $\mod A$. Moreover, we have $$\{\Fac M,\Fac M^\prime\}=\{\Fac U,\prescript{\bot}{}{(\tau U)}\cap Q^\bot\}.$$
In particular, either $\Fac M\subsetneqq \Fac M^\prime$ or
$\Fac M^\prime\subsetneqq\Fac M$ holds.
\end{theorem}

\begin{definition}[Left and right mutation]
Keep the notations in Theorem \ref{thmair}. We call the operation $(M,P)\mapsto(M^\prime,P^\prime)$ a \emph{mutation} of $(M,P)$. If $\Fac M\subsetneqq \Fac M^\prime$ holds, we call $(M^\prime, P^\prime)$ a {\em right mutation} of $(M,P)$. If $\Fac M^\prime \subsetneqq \Fac M$ holds, we call $(M^\prime, P^\prime)$ a {\em left mutation} of $(M,P)$.
\end{definition}

\subsection{Functorially finite torsion classes}

A \emph{torsion pair} $(\mathcal{T},\mathcal{F})$ in  $\mod A$ is a pair of subcategories of $\mod A$ satisfying that
\begin{itemize}
\item[(i)] $\Hom_A(T,F)=0$ for any $T\in \mathcal{T}$ and $F\in \mathcal{F}$;
\item[(ii)] for any module $X\in \mod A$, there exists a short exact sequence
 $$0\rightarrow X^{\mathcal{T}}\rightarrow X\rightarrow X^{\mathcal{F}}\rightarrow0$$ with $X^{\mathcal{T}}\in\mathcal{T}$ and $X^{\mathcal{F}}\in \mathcal{F}$. Thanks to the condition (i), such a sequence is unique up to isomorphisms. This short exact sequence is called the \emph{canonical sequence} of $X$ with respect to $(\mathcal{T},\mathcal{F})$.
\end{itemize}

The subcategory $\mathcal T$ (resp., $\mathcal F$) in a torsion pair $(\mathcal T,\mathcal F)$ is called a \emph{torsion class} (resp., \emph{torsionfree class}) of $\mod A$. A torsion class $\mathcal T$ is said \emph{functorially finite} if it is functorially finite as a subcategory of $\mod A$, that is, $\mathcal T$ is both contravariantly finite and covariantly finite in $\mod A$, cf. \cite{AS_1981}*{Page 432}.

\begin{theorem}[\cite{air}*{{Proposition 1.2 (b) and Theorem 2.7}}]\label{proorder}
There is a well-defined map $\Psi$ 
from $\tau$-rigid pairs to functorially finite torsion classes in $\mod A$
given by $(M,P)\mapsto \Fac M$.
Moreover, $\Psi$ is a bijection if we restrict it to basic $\tau$-tilting pairs.
\end{theorem}

\begin{proposition}[\cite{air}*{Proposition 2.9 and Theorem 2.10}]\label{prominmax}
 Let $(U,Q)$ be a basic $\tau$-rigid pair and $\mathcal T$ a functorially finite torsion class in $\mod A$. Then
 \begin{itemize}
 \item[(i)]  $\Fac U$ and  $\prescript{\bot}{}({\tau U})\cap Q^\bot$ are functorially finite torsion classes in $\mod A$;
 \item[(ii)]  $\Fac U\subseteq \mathcal T\subseteq \prescript{\bot}{}({\tau U})\cap Q^\bot$ if and only if $(U,Q)$ is a direct summand of the basic $\tau$-tilting pair $\Psi^{-1}(\mathcal T)$, where $\Psi$ is the bijection given in Theorem \ref{proorder}.
 \end{itemize}
\end{proposition}

\begin{definition}[Bongartz completion and Bongartz co-completion] Let $(U,Q)$ be a basic $\tau$-rigid pair in $\mod A$ and  $\Psi$ the bijection in Theorem \ref{proorder}.
\begin{itemize}
\item[(i)] The basic $\tau$-tilting pair $\Psi^{-1}(\Fac U)$ is
called the \emph{Bongartz co-completion} of $(U,Q)$.
 \item[(ii)] The basic $\tau$-tilting pair $\Psi^{-1}(\prescript{\bot}{}({\tau U})\cap Q^\bot)$ is called the \emph{Bongartz completion} of $(U,Q)$.
\end{itemize}
\end{definition}

\subsection{$g$-vectors and $c$-vectors in $\tau$-tilting theory}
Let $A = \bigoplus\limits_{i=1}^n P_i$ be a decomposition of $A$, where $P_i$ is an indecomposable projective $A$-module.  For $M\in\mod A$, we take a minimal projective presentation of $M$
\[\xymatrix{\bigoplus\limits_{i=1}^nP_i^{b_i}\ar[r]&\bigoplus\limits_{i=1}^nP_i^{a_i}\ar[r]&M\ar[r]&0}.\]
The column vector $\mathbf{ g}(M):=(a_1-b_1,\cdots,a_n-b_n)^\mathrm{T}\in\mathbb Z^n$ is called the \emph{$g$-vector} of $M$, {where $X^{\rm T}$ denotes the transpose of a matrix $X$.}
For a $\tau$-rigid pair $(M,P)$, its \emph{$g$-vector} is defined to be $$\mathbf{ g}(M,P):=\mathbf{ g}(M)-\mathbf{ g}(P) \in \ZZ^n.$$

When we write a basic $\tau$-tilting pair $(M,P)$ as the form $(M,P)=\bigoplus\limits_{i=1}^n(M_i,Q_i)$, we always mean that each $(M_i,Q_i)$ is an indecomposable $\tau$-rigid pair, that is, $M_i\oplus Q_i$ is indecomposable in $\mod A$.

For a basic $\tau$-tilting pair $(M,P)=\bigoplus\limits_{i=1}^n(M_i,Q_i)$, its \emph{$G$-matrix} is the following integer matrix 

$$G_{(M,P)}:=(\mathbf{ g}(M_1,Q_1),\cdots,\mathbf{ g}(M_n,Q_n)),$$ whose columns are the $g$-vectors $\mathbf{ g}(M_1,Q_1),\cdots,\mathbf{ g}(M_n,Q_n)$. Thanks to the theorem below, there exists another integer matrix $C_{(M,P)}$ associated to $(M,P)$, which is given by
$C_{(M,P)}:=(G_{(M,P)}^\mathrm{T})^{-1}$. The matrix $C_{(M,P)}$ is called the \emph{$C$-matrix} of $(M,P)$ and its column vectors are called \emph{$c$-vectors}.

\begin{theorem}[\cite{air}*{Theorem 5.1}] Let $(M,P)=\bigoplus\limits_{i=1}^n(M_i,Q_i)$ be a basic $\tau$-tilting pair in $\mod A$. Then $\mathbf{ g}(M_1,Q_1),\cdots,\mathbf{ g}(M_n,Q_n)$ form a $\mathbb Z$-basis of $\mathbb Z^n$.
\end{theorem}

\begin{theorem}[\cite{air}*{Theorem 5.5}]
\label{ginjection}
The map $(M,P)\mapsto{\bf g}(M,P)$ gives an injection from the set of isomorphism classes of $\tau$-rigid pairs in $\mod A$ to $\mathbb Z^n$.
\end{theorem}

\section{Cluster algebras}

\subsection{Cluster algebras}\label{sec31}
We recall that a \emph{semifield} $\mathbb P$ is an abelian multiplicative group equipped with an addition $\oplus$ which is distributive over the multiplication. Now we give two important examples of semifield: tropical semifield and universal semifield.

Let Trop$(u_1,\dots, u_\ell)$ be the free (multiplicative) abelian group generated by $u_1,\dots,u_\ell$. Then $\text{Trop}(u_1,\dots,u_{\ell})$ is a semifield with the addition defined by
\begin{align}
\prod_{j=1}^\ell u_j^{a_j} \oplus \prod_{j=1}^{\ell} u_j^{b_j}=\prod_{j=1}^{\ell} u_j^{\min(a_j,b_j)}.\nonumber
\end{align}
We call it the \emph{tropical semifield} of $u_1,\dots,u_\ell$.

Let $\mathbb Q_{\text{sf}}(u_1,\dots,u_{\ell})$ be the set of non-zero rational functions in $u_1,\dots,u_{\ell}$ which have subtraction-free expressions. Then $\mathbb Q_{\text{sf}}(u_1,\dots,u_{\ell})$ is a semifield with the usual multiplication and addition. We call it the \emph{universal semifield} of $u_1,\dots,u_{\ell}$.

Clearly, for any semifield $\PP$ and $p_1, \dots, p_{\ell}\in\PP$, there exists a unique semifield homomorphism $\pi: \mathbb Q_{\text{sf}}(u_1,\dots,u_{\ell})\rightarrow \mathbb P$ induced by $\pi:u_i\mapsto p_i$ for $i=1,\cdots,\ell$. The image $\pi(F(u_1, \dots, u_\ell))$ of $F(u_1, \dots, u_\ell)\in\mathbb Q_{\text{sf}}(u_1,\dots,u_{\ell})$ is called
the \emph{evaluation} of $F(y_1,\dots,y_\ell )$ at $(p_1,\cdots,p_\ell)$ and we denote
$$F|_{\PP}(p_1, \dots, p_{\ell}):=\pi(F(u_1, \dots, u_\ell)).$$

Let $B=(b_{ij})$ be an $n \times n$ integer matrix. We say that $B$ is {\em skew-symmetrizable} if there exists a positive integer diagonal matrix $S=\mathrm{diag}(s_1,\dots,s_n)$ such that $SB$ is skew-symmetric. Such an $S$ is called a \emph{skew-symmetrizer} of $B$.

A \emph{(labeled) $Y$-seed} of rank $n$ over a semifield $\PP$ is a pair $(\mathbf{y}, B)$, where
\begin{itemize}
\item $\mathbf{y}=(y_1, \dots, y_n)$ is an $n$-tuple of elements of $\mathbb{P}$. Elements in $\mathbf{ y}$ are called \emph{coefficients}.
\item $B=(b_{ij})$ is an $n \times n$ skew-symmetrizable integer matrix, which is called the {\em exchange matrix} of $(\mathbf{y}, B)$.
\end{itemize}

We fix a positive integer $n$ and a semifield $\PP$. Let $\mathbb{ZP}$ be the group ring of $\mathbb{P}$ as a multiplicative group. Since $\mathbb{ZP}$ is a domain (\cite{fzi}*{Section 5}), its total quotient ring is a field $\mathbb{Q}\mathbb P$. Let $\mathcal{F}$ be the field of the rational functions in $n$ indeterminates with coefficients in $\mathbb{Q}\mathbb P$.

A \emph{(labeled) seed} of rank $n$ over $\PP$ is a triplet $(\mathbf{x}, \mathbf{y}, B)$, where
\begin{itemize}
\item $\mathbf{x}=(x_1, \dots, x_n)$ is an $n$-tuple of elements of $\mathcal F$ such that they form a free generating set of $\mathcal F$. $\mathbf{x}$ is called a {\em cluster} and elements in $\mathbf{ x}$ are called \emph{cluster variables}.
\item $(\mathbf{y}, B)$ forms a $Y$-seed of rank $n$ over $\PP$, which is called the {\em underlying $Y$-seed} of $(\mathbf{x}, \mathbf{y}, B)$.
\end{itemize}

For an integer $b$, we denote $[b]_+=\max(b,0)$. Clearly, we have $b=[b]_+-[-b]_+$.

\begin{definition}[Matrix mutation] Let $A=(a_{ij})$ be an $m\times n$ integer matrix with $m\geq n>0$ and $k\in\{1,\cdots,n\}$. The {\em mutation} $\mu_k(A)$ of $A$ in direction $k$ is the new matrix $A^\prime=(a_{ij}^\prime)$ given by
\begin{eqnarray}
\label{eq:matrix-mutation}
a'_{ij}=\begin{cases}-a_{ij} ,&\text{if $i=k$ or $j=k$;} \\
a_{ij}+\left[ a_{ik}\right] _{+}a_{kj}+a_{ik}\left[ -a_{kj}\right]_+ ,&\text{otherwise.}\nonumber
\end{cases}
\end{eqnarray}
\end{definition}
It can be checked that if the upper $n\times n$ submatrix of $A_{m\times n}$ is skew-symmetrizable, then $$\mu_k(\mu_k(A))=A.$$

Let $(\mathbf{y}, B)$ be a $Y$-seed of rank $n$ over $\PP$ and $k \in\{1,\dots, n\}$. The \emph{$Y$-seed mutation} $\mu_k$ in direction $k$ transforms $(\mathbf{y}, B)$ to a new $Y$-seed $\mu_k(\mathbf{y}, B)=(\mathbf{y'}, B')$ defined as follows:
\begin{itemize}
\item $B^\prime$ is the matrix mutation of $B$ in direction $k$, that is, $B^\prime=\mu_k(B)$;
\item The coefficients in $\mathbf{y'}=(y'_1, \dots, y'_n)$ are given by
\begin{eqnarray}
\label{eq:y-mutation}
y'_j=
\begin{cases}
y_{k}^{-1}, &\text{if $j=k$;} \\
y_j y_k^{[b_{kj}]_+}(y_k \oplus 1)^{-b_{kj}} ,&\text{otherwise.}
\end{cases}\nonumber
\end{eqnarray}
\end{itemize}

Let $(\mathbf{x}, \mathbf{y}, B)$ be a seed of rank $n$ over $\PP$ and $k \in\{1,\dots, n\}$. The \emph{seed mutation} $\mu_k$ in direction $k$ transforms $(\mathbf{x}, \mathbf{y}, B)$ to a new seed $\mu_k(\mathbf{x}, \mathbf{y}, B)=(\mathbf{x'}, \mathbf{y'}, B')$ defined as follows:
\begin{itemize}
\item $(\mathbf{y'}, B')$ is obtained from $(\mathbf{y}, B)$ by $Y$-seed mutation in direction $k$.
\item The cluster variables in $\mathbf{x'}=(x'_1, \dots, x'_n)$ are given by
\begin{align} 
x'_j=\begin{cases}\dfrac{y_k\mathop{\prod}\limits_{i=1}^{n} x_i^{[b_{ik}]_+}+\mathop{\prod}\limits_{i=1}^{n} x_i^{[-b_{ik}]_+}}{(y_k\oplus 1)x_k}, &\text{if $j=k$;}\\
x_j, &\text{otherwise.}\nonumber
\end{cases}
\end{align}
\end{itemize}

It can be checked that the seed mutation $\mu_k$ is an involution, that is, $\mu_k^2({\bf x},{\bf y}, B)=({\bf x},{\bf y}, B)$. This in particular implies that the $Y$-seed mutation is also an involution.

Let $\mathbb{T}_n$ be the \emph{$n$-regular tree} whose edges are labeled by the numbers $1, \dots, n$ such that the $n$ edges emanating from each vertex have different labels. We write
\begin{xy}(0,1)*+{t}="A",(10,1)*+{t'}="B",\ar@{-}^k"A";"B" \end{xy}
to indicate that the two vertices $t,t'\in \mathbb{T}_n$ are joined by an edge labeled by $k$.

\begin{definition}[Patterns]
(i) A {\em $Y$-pattern}  $\Sigma^Y=\{t\mapsto \Sigma_t^Y\}_{t\in\mathbb T_n}$
of rank $n$ over $\PP$ is an assignment of a $Y$-seed $\Sigma_t^Y=( \mathbf{y}_t,B_t)$ of rank $n$ over $\PP$ to every vertex $t\in \mathbb{T}_n$ such that
for any edge \begin{xy}(0,1)*+{t}="A",(10,1)*+{t'}="B",\ar@{-}^k"A";"B" \end{xy} of $\mathbb T_n$, we have $\Sigma_{t^\prime}^Y=\mu_k(\Sigma_t^Y)$.

(ii) A \emph{cluster pattern} $\Sigma=\{t\mapsto \Sigma_t\}_{t\in\mathbb T_n}$ of rank $n$ over $\PP$ is an assignment of a seed $\Sigma_t=(\mathbf{x}_t, \mathbf{y}_t,B_t)$ of rank $n$ over $\PP$ to every vertex $t\in \mathbb{T}_n$ such that
for any edge \begin{xy}(0,1)*+{t}="A",(10,1)*+{t'}="B",\ar@{-}^k"A";"B" \end{xy} of $\mathbb T_n$, we have $\Sigma_{t^\prime}=\mu_k(\Sigma_t)$.
\end{definition}

Clearly each cluster pattern induces a $Y$-pattern by forgetting the cluster part of each seed. The resulting $Y$-pattern is called the {\em underlying $Y$-pattern} of this cluster pattern.

 For a seed $\Sigma_t=(\mathbf{x}_t, \mathbf{y}_t,B_t)$, we always denote by
$$
\mathbf{x}_t=(x_{1;t},\dots,x_{n;t}),\ \mathbf{y}_t=(y_{1;t},\dots,y_{n;t}),\ B_t=(b_{ij;t}).
$$

Sometimes we would like to choose a vertex $t_0$ as the \emph{rooted vertex} of $\mathbb T_n$. The seed at the rooted vertex $t_0$ would be called an \emph{initial seed}. In this case, we will also use $(\mathbf{ x},\mathbf{ y},B)$ to denote the initial seed $\Sigma_{t_0}$, where
\begin{align}
\mathbf{x}=\mathbf{x}_{t_0}=(x_1,\dots,x_n),\ \mathbf{y}=\mathbf{y}_{t_0}=(y_1,\dots,y_n),\ B=B_{t_0}=(b_{ij}).\nonumber
\end{align}

\begin{definition}[Cluster algebra and upper cluster algebra]
Let $\Sigma=\{t\mapsto \Sigma_t\}_{t\in\mathbb T_n}$ be a cluster pattern of rank $n$ over $\mathbb P$. 
\begin{itemize}
    \item [(i)] The \emph{cluster algebra} $\Acal$ associated with $\Sigma$ is the $\ZZ\PP$-subalgebra of $\Fcal$ given by
 $$\mathcal A=\mathbb{ZP}[x_{1;t},\cdots,x_{n;t}|\;t\in\mathbb T_n].$$
 \item[(ii)] The {\em upper cluster algebra} $\mathcal U$ associated with $\Sigma$ is the $\ZZ\PP$-subalgebra of $\Fcal$ given by the following intersections.
 $$\mathcal U=\bigcap\limits_{t\in\mathbb T_n}\mathbb{ZP}[x_{1;t}^{\pm 1},\cdots,x_{n;t}^{\pm 1}].$$
\end{itemize}
We call $n$ the \emph{rank} of the (upper) cluster algebra. 
 \end{definition}
We always use $\mathcal{A}(B,t_0)$ to denote a cluster algebra with initial exchange matrix $B$ at the rooted vertex $t_0\in\mathbb T_n$.

\begin{example}[Cluster algebra of Type $A_2$] \label{A2} Take $n=2$, then the $2$-regular tree $\TT_2$ is given as follows:
\begin{align}\label{A2tree}
\begin{xy}
(-10,0)*+{\dots}="a",(0,0)*+{t_0}="A",(10,0)*+{t_1}="B",(20,0)*+{t_2}="C", (30,0)*+{t_3}="D",(40,0)*+{t_4}="E",(50,0)*+{t_5}="F", (60,0)*+{\dots}="f"
\ar@{-}^{1}"a";"A"
\ar@{-}^{2}"A";"B"
\ar@{-}^{1}"B";"C"
\ar@{-}^{2}"C";"D"
\ar@{-}^{1}"D";"E"
\ar@{-}^{2}"E";"F"
\ar@{-}^{1}"F";"f"
\end{xy}.
\end{align}
We set the initial exchange matrix $B=\begin{bmatrix}
 0 & 1 \\
 -1 & 0
\end{bmatrix}
$ at the vertex $t_0$. In this case, we only have finitely many coefficients and cluster variables, which are given in Table \ref{A2seed} \cite{fziv}*{Example 2.10}.
\begin{table}[ht]
\begin{equation*}
\begin{array}{|c|cc|cc|}
\hline
&&&&\\[-4mm]
t& \hspace{25mm}\yy_t &&& \xx_t \hspace{30mm}\\
\hline
&&&&\\[-3mm]
0 &y_1 & y_2& x_1& x_2 \\[1mm]
\hline
&&&&\\[-3mm]
1& y_1(y_2\oplus 1)& \dfrac{1}{y_2} & x_1& \dfrac{x_1y_2+1}{(y_2\oplus 1)x_2} \\[3mm]
\hline
&&&&\\[-3mm]
2& \dfrac{1}{y_1(y_2\oplus 1)} & \dfrac{y_1y_2\oplus y_1\oplus 1}{y_2} & \dfrac{x_1y_1y_2 + y_1+ x_2}{(y_1y_2\oplus y_1\oplus 1)x_1x_2} & \dfrac{x_1y_2+1}{(y_2\oplus 1)x_2} \\[3mm]
\hline
&&&&\\[-3mm]
3& \dfrac{y_1\oplus1}{y_1y_2} & \dfrac{y_2}{y_1y_2\oplus y_1\oplus 1} & \dfrac{x_1y_1y_2+y_1+x_2}{(y_1y_2\oplus y_1\oplus 1)x_1x_2} & \dfrac{y_1+x_2}{x_1(y_1\oplus 1)} \\[3mm]
\hline
&&&&\\[-2mm]
4& \dfrac{y_1y_2}{y_1\oplus 1} &\dfrac{1}{y_1} & x_2 & \dfrac{y_1+x_2}{x_1(y_1\oplus 1)} \\[3mm]
\hline
&&&&\\[-2mm]
5& y_2 & y_1 & x_2 & x_1\\[1mm]
\hline
\end{array}
\end{equation*}
\caption{Coefficients and cluster variables in type~$A_2$: general coefficients\label{A2seed}}
\end{table}

Therefore, we have
\begin{align*}
\Acal(B,t_0)=\ZZ\PP\left[x_1,x_2,\dfrac{x_1y_2+1}{(y_2\oplus 1)x_2},\dfrac{x_1y_1y_2+y_1+x_2}{(y_1y_2\oplus y_1\oplus 1)x_1x_2}, \dfrac{y_1+x_2}{x_1(y_1\oplus 1)}\right].
\end{align*}
\end{example}

 Let $\Sigma_t=(\xx_t,\yy_t,B_t)$ be a seed and $\sigma$ a permutation on $\{1,\cdots,n\}$. We use $\sigma(\Sigma_t)$ to denote a new seed given as follows:
\begin{align*}
\sigma(\Sigma_t)&=(\sigma\xx_t,\sigma\yy_t,\sigma B_t),\\
    \sigma\xx_t&=(x_{\sigma(1);t},\dots,x_{\sigma(n);t}),\\
    \sigma\yy_t&=(y_{\sigma(1);t},\dots,y_{\sigma(n);t}),\\
    \sigma B_t&=(b'_{ij}),\quad b'_{ij}=b_{\sigma(i)\sigma(j);t}.
\end{align*}

Two seeds $\Sigma_t$ and $\Sigma_{t^\prime}$ are said to be \emph{equivalent} if there exists a permutation $\sigma$ such that $\Sigma_{t^\prime}=\sigma(\Sigma_t)$. Similarly, we can define the notation of the equivalence between $Y$-seeds.

\begin{remark}
Notice that the action of a permutation on a seed is compatible with mutations, that is, we have
\begin{align}
    \mu_{k}(\sigma(\Sigma_t))=\sigma\mu_{\sigma(k)}(\Sigma_t).\nonumber
\end{align}
In particular, the two cluster algebras given by two equivalent seeds have the same set of cluster variables.
\end{remark}

\begin{theorem} [\cite{cl2}*{Proposition 3 (ii)}] \label{non-labeled-cluster-thm}
Let $\Sigma_t$ and $\Sigma_{t^\prime}$ be two seeds of a cluster algebra $\mathcal A$. If there exists a permutation $\sigma$ such that $\xx_{t^\prime}=\sigma\xx_{t}$, then $\Sigma_{t^\prime}=\sigma(\Sigma_t)$.
\end{theorem}

\begin{proposition}[\cite{cl2}*{Prposition 3 (i)}] \label{proindepen}
Let $\mathcal A_1$ be a cluster algebra with coefficient semifield $\mathbb P_1$ and $\mathcal A_2$ a cluster algebra with coefficient semifield $\mathbb P_2$. Denote by $(\xx_t(k),\yy_t(k), B_t(k))$ the seed of $\Acal_k$ at $t\in\TT_n$, $k= 1,2$. If $\Acal_1$ and $\Acal_2$ have the same initial exchange matrix at the rooted vertex $t_0$, then $x_{i;t}(1) =x_{j;t'}(1)$ if and only if $x_{i;t}(2) =x_{j;t'}(2)$, where $t, t'\in\TT_n$ and $i, j\in \{1,2,\cdots, n\}$.
\end{proposition}

\begin{definition}[Exchange graph]
The \emph{exchange graph} $\Gamma$ of a cluster algebra $\Acal$ is the graph whose vertices correspond to the seeds of $\mathcal A$ up to seed equivalence and whose edges correspond to the seed mutations.
\end{definition}
By Proposition \ref{proindepen} and Theorem \ref{non-labeled-cluster-thm}, we know that the exchange graph $\Gamma$ of a cluster algebra $\mathcal A$ is independent of the choice of the coefficient semifield. It only depends on the initial exchange matrix.

\begin{theorem}[\cite{cl2}*{Theorem 10}] \label{thmgraph}
The seeds of a cluster algebra $\Acal$ whose clusters contain particular cluster variables form a connected subgraph of the
exchange graph of $\Acal$.
\end{theorem}

\subsection{$c$-vectors and $g$-vectors in cluster algebras} In this subsection, we recall the $c$-vectors and $g$-vectors in cluster algebras.

If the coefficient semifield of a cluster algebra $\mathcal A$ is a tropical semifield, then we say that $\mathcal A$ is a  {\em geometric cluster algebra}.

\begin{definition}[Coefficient matrix and extended exchange matrix]
Let $\Acal$ be a geometric cluster algebra whose coefficient semifield is ${\rm Trop}(u_1,\cdots,u_m)$. We know that each coefficient variable $y_{k;t}$ of $\Acal$ can be written as a Laurent monomial in $u_1,\cdots,u_m$, say $y_{k;t}=u_1^{c_{1k}^t}\cdots u_m^{c_{mk}^t}$. We call
$C_t:=(c_{ij}^t)_{m\times n}$ the \emph{coefficient matrix} at $t$ and  $\tilde B_t=\begin{bmatrix}B_t\\ C_t\end{bmatrix}$ the {\em extended exchange matrix} at $t$.
\end{definition}

\begin{proposition}\label{procomat}
Let $\mathcal A$ be a geometric cluster algebra and $\Sigma_{t_1}, \Sigma_{t_2}$ two seeds of $\mathcal A$. Let $C_{t_1}=({\bf c}_1,\cdots,{\bf c}_n)$ and $C_{t_2}$ be the coefficient matrices of $\Sigma_{t_1}$ and $\Sigma_{t_2}$, respectively. If there exists a permutation $\sigma$ such that $\Sigma_{t_2}=\sigma(\Sigma_{t_1})$, then $C_{t_2}=({\bf c}_{\sigma(1)},\cdots,{\bf c}_{\sigma(n)})$.
\end{proposition}
\begin{proof}
By $\Sigma_{t_2}=\sigma(\Sigma_{t_1})$, we have ${\bf y}_{t_2}=\sigma{\bf y}_{t_1}$. Then by the definition of coefficient matrix, we have $C_{t_2}=({\bf c}_{\sigma(1)},\cdots,{\bf c}_{\sigma(n)})$.
\end{proof}

\begin{definition}[Extended cluster and frozen variables]
Let $\Acal$ be a geometric cluster algebra whose coefficient semifield is ${\rm Trop}(u_1,\cdots,u_m)$. For a seed $\Sigma_t=({\bf x}_t,{\bf y}_t,B_t)$ of $\mathcal A$, we usually use $(\tilde {\bf x}_t, \tilde B_t)$ to denote the seed $\Sigma_t$, where $\tilde B_t$ is the extended exchange matrix of $\Sigma_t$ and $$\tilde {\bf x}_t=(x_{1;t},\cdots,x_{n;t};x_{n+1;t},\cdots,x_{n+m;t})$$ with $x_{n+i;t}=u_i$ for $i=1,\cdots,m$. We call $\tilde {\bf x}_t$ an {\em extended cluster} and $x_{n+1;t},\cdots,x_{n+m;t}$ the {\em frozen variables}.
\end{definition}

\begin{proposition}[\cite{fzi}*{Proposition 5.8}] Let $\mathcal A$ be a geometric cluster algebra. Then for any edge $t^{~\underline{\quad k \quad}}~ t^{\prime}$ in $\mathbb T_n$, we have $\tilde B_{t^\prime}=\mu_k(\tilde B_t)$, where $\tilde B_t$ and $\tilde B_{t^\prime}$ are the extended exchange matrices at $t$ and $t^\prime$, respectively.
\end{proposition}

Let $\mathcal A$ be a geometric cluster algebra whose coefficient semifield is  $\trop(u_1,\cdots,u_m)$. We say that $\mathcal A$ is a cluster algebra with \emph{principal coefficients} at seed $\Sigma_{t_0}=(\mathbf{ x},\mathbf{ y}, B)$, if the coefficient matrix $C_{t_0}$ of $\Sigma_{t_0}$ is the identity matrix $I_n$, that is, $m=n$ and $\mathbf{y}=(y_1,\dots,y_n)=(u_1,\cdots,u_n)$ hold.

\begin{definition}[$C$-matrix and $c$-vectors] Let $\Acal(B, t_0)$ be a cluster algebra with
principal coefficients at seed $\Sigma_{t_0}$. Then the coefficient matrix $C_t=C_t^{B;t_0}$ at $t\in\mathbb T_n$ is called a {\em $C$-matrix} of $B$ and its columns are called {\em $c$-vectors} of $B$.
\end{definition}

\begin{theorem}[\cite{fziv}*{Proposition 6.1}]\label{thmfziv}
Let $\Acal(B,t_0)$ be a cluster algebra with principal coefficients at seed $\Sigma_{t_0}$. Then each cluster variable $x_{k;t}$ is a homogeneous Laurent polynomial in $x_1,\dots,x_n,y_1,\dots,y_n$ under the following $\ZZ^n$-grading:
\begin{align}\label{grading}
\deg x_{i}=\ee_i,\quad \deg y_{i}=-\mathbf{b}_i,
\end{align}
where $\ee_i$ is the $i$th canonical basis of $\ZZ^n$ and $\mathbf{b}_i$ is the $i$th column vector of $B$.
\end{theorem}
\begin{definition}[$G$-matrix and $g$-vectors]
Keep the notations in Theorem \ref{thmfziv}. The degree $\deg x_{k;t}$ of $x_{k;t}$ under the $\ZZ^n$-grading \eqref{grading} is called the \emph{$g$-vector} of $x_{k;t}$ and we denote it by $$\gg_{k;t}^{B;t_0}=\gg_{k;t}=\deg x_{k;t}.$$ The matrix $G_t^{B;t_0}=(\gg_{1;t},\dots,\gg_{n;t})$ is called a \emph{$G$-matrix} of $B$.
\end{definition}
Notice that both the $C$-matrix $C_{t}^{B;t_0}$ and the $G$-matrix $G_t^{B;t_0}$ are uniquely determined by the triple $(B,t_0,t)$.

\begin{theorem}[Sign-coherence, \cite{GHKK}] The following two statements hold.
\label{thm:signs-ci}
\begin{itemize}
\item[(i)] Each column vector of a $C$-matrix is either a non-negative vector or a non-positive vector.

\item[(ii)] Each row vector of a $G$-matrix is either a non-negative vector or a non-positive vector.
\end{itemize}
\end{theorem}

\begin{theorem}[\cites{nz,cl2018}]\label{thmcg} Let $B$ be an $n\times n$ skew-symmetrizable matrix and $S$ a skew-symmetrizer of $B$. Then for any vertices $t_0, t\in\mathbb T_n$, we have
$SC_t^{B;t_0}S^{-1}(G_t^{B;t_0})^\mathrm{T}=I_n$. In particular, $\det C_t^{B;t_0}=\det G_t^{B;t_0}=\pm1$.
\end{theorem}

\subsection{$g$-pairs of clusters}
In this subsection, we briefly recall $g$-pairs of clusters introduced in \cite{cl2} with the aim to study $d$-vectors of cluster algebras. We will use its properties later. We denote by $[\xx_t]$ the cluster $\xx_t=(x_{1;t},\cdots,x_{n;t})$ up to permutations, that is, $[\mathbf{ x}_t]=\{x_{1;t},\cdots,x_{n;t}\}$.

\begin{definition}\label{def:g-pair}
Let $\Acal$ be a cluster algebra with initial seed $\Sigma_{t_0}$ and $U$ a subset of $[\mathbf{ x}_{t_0}]$. A pair of clusters $([\mathbf{ x}_t],[\mathbf{ x}_{t^\prime}])$
is called a \emph{$g$-pair} associated with $U$, if it satisfies the
following two conditions.
\begin{itemize}
 \item [(a)] $U$ is a subset of $[\xx_{t'}]$;
 \item[(b)] The $i$th row of $(G_{t^\prime}^{B_{t_0};t_0})^{-1}G_t^{B_{t_0};t_0}$ is a non-negative vector for any $i$ such that $x_{i;t^\prime}\notin U$.
\end{itemize}
\end{definition}

\begin{theorem}[\cite{cl2}*{Theorem 8,9}]\label{thmgpair}
Let $\Acal$ be a cluster algebra with initial seed $t_0$. Then for any subset $U\subset[\mathbf{ x}_{t_0}]$ and any cluster $[\mathbf{ x}_t]$, there exists a unique cluster $[\mathbf{ x}_{t^\prime}]$ such that $([\mathbf{ x}_t],[\mathbf{ x}_{t^\prime}])$ is a $g$-pair associated with $U$.
\end{theorem}
We refer to Theorem \ref{thmgpair} the {\it enough $g$-pairs property} of cluster algebras.

\begin{remark}
The authors in \cite{cl2} proved the above theorem for principal coefficients algebras. However, since the notation of $g$-pairs only depends on $G$-matrices, it can be naturally extended to cluster algebras with coefficients in a general semifield.
\end{remark}

The following property of $g$-pairs was given in \cite{cao}.

\begin{theorem}[\cite{cao}*{Corollary 7.18}] \label{thmgcompletion}
Let $\Acal$ be a cluster algebra with initial seed $t_0$ and $U$ a subset of $[\mathbf{ x}_{t_0}]$. Then $([\mathbf{ x}_t], [\mathbf{ x}_{t^\prime}])$ is a $g$-pair associated with $U$ if and only if the following two conditions hold.
\begin{itemize}
 \item [(a)] $U$ is a subset of $[\xx_{t'}]$;
 \item[(b)] The $i$th row of the $G$-matrix $G_t^{B_{t^\prime};t^\prime}$ is a non-negative vector for any $i$ such that $x_{i;t^\prime}\notin U$.
\end{itemize}

\end{theorem}

\section{Bongartz completion via $c$-vectors}
\subsection{Bongartz completion in $\tau$-tilting theory via $c$-vectors} We write $\langle-,-\rangle$ for the canonical inner product in $\mathbb R^n$, that is, $\langle{\bf a},{\bf b}\rangle=a_1b_1+\cdots+a_nb_n$ for ${\bf a},{\bf b}\in\mathbb R^n$. For a module $M\in\mod A$, we use $[M]\in\mathbb N^n$ to denote the dimensional vector of $M$.

Let $\theta$ be a vector in $\mathbb R^n$ and $M$ a module in $\mod A$. We say that $M$ is {\em $\theta$-stable} (respectively, {\em $\theta$-semistable}) if $\langle\theta,[M]\rangle=0$ and $\langle\theta,[L]\rangle<0$ (respectively, $\langle\theta,[L]\rangle\leq0$) for any non-zero proper submodule $L$ of $M$.

In the following theorem, we summarize some known results from \cites{asai,tref}, mainly from \cite{asai}*{Proposition 2.17 and Subsection 3.4} and \cite{tref}*{Lemma 3.10 and Theorem 3.13}.
\begin{theorem}[\cites{asai,tref}]\label{thmtref}
Let  $(M,P)=\bigoplus\limits_{i=1}^n(M_i,Q_i)$ be a basic $\tau$-tilting pair in $\mod A$. Suppose that the mutation $(M^\prime,P^\prime)$ of $(M,P)$ in direction $r$ is a right mutation and put $(U,Q)=\bigoplus\limits_{i\neq r}(M_i,Q_i)$. Then the following statements hold.
\begin{itemize}
\item[(i)] There exists a unique brick $B_r$ up to isomorphisms such that $B_r$ belongs to $U^\bot \cap  \prescript{\bot}{}{(\tau U)}\cap Q^\bot$.

\item[(ii)] The $r$th column vector of the $C$-matrix $C_{(M,P)}$ (resp., $C_{(M^\prime,P^\prime)}$) is given by $-[B_r]$ (resp., $[B_r]$), where $[B_r]$ denotes the dimensional vector of the brick $B_r$.
\end{itemize}
\end{theorem}

\begin{remark}
(i) It is known from \cite{bst}*{Proposition 3.13} or \cite{yu2018}*{Theorem 1.4} that
the subcategory $U^\bot \cap  \prescript{\bot}{}{(\tau U)}\cap Q^\bot$ in Theorem \ref{thmtref} is the category of $\mathbf{ g}(U,Q)$-semistable modules in $\mod A$. In this case, the brick $B_r$ is the unique $\mathbf{ g}(U,Q)$-stable module in $\mod A$ up to isomorphisms, cf. \cite{bst}*{Theorem 3.14}.

(ii) For more about the links between $c$-vectors and dimension vectors of representations, we refer the reader to the earlier literatures \cites{nagao,sh13,najera,najera13}.
\end{remark}

\begin{lemma}\label{lemright}
Let $(M,P)=\bigoplus\limits_{i=1}^n(M_i,Q_i)$ be a basic $\tau$-tilting pair in $\mod A$ and $(M^\prime,P^\prime)$ the mutation of $(M,P)$ in direction $r$. Then the following statements are equivalent.
\begin{itemize}
\item[(i)] $(M^\prime, P^\prime)$ is a right mutation of $(M,P)$;
\item[(ii)] $\Fac M\subsetneqq\Fac M^\prime$;
\item[(iii)] The $r$th column vector of the $C$-matrix $C_{(M,P)}$ is a non-positive vector.
\end{itemize}
\end{lemma}
\begin{proof}
The assertion follows from the definition of right mutation and  Theorem \ref{thmtref}.
\end{proof}

\begin{lemma}[\cite{air}*{Theorem 2.35}] \label{lemair}
Let $(M,P)$ and $(N,R)$ be two basic $\tau$-tilting pairs with $\Fac M\subsetneqq\Fac N$. Then there exists a right mutation $(M^\prime,P^\prime)$ of $(M,P)$ such that $\Fac M^\prime\subseteq\Fac N$.

\end{lemma}

\begin{theorem}\label{thm-tau}
Let $(U,Q)$ be a basic $\tau$-rigid pair and $(M,P)=\bigoplus\limits_{i=1}^n(M_i,Q_i)$ a basic $\tau$-tilting pair in $\mod A$. Then $(M,P)$ is the Bongartz completion of $(U,Q)$ if and only if the following two conditions hold.
\begin{itemize}
\item[(a)] $(U,Q)$ is a direct summand of $(M,P)$, say $(U,Q)=\bigoplus\limits_{i=1}^s(M_i,Q_i)$, where $s\leq n$;
\item[(b)] The $c$-vectors in the $C$-matrix $C_{(M,P)}$ indexed by $s+1,\cdots,n$ are non-negative vectors.
\end{itemize}
\end{theorem}
\begin{proof}
$\Longrightarrow$: Suppose that $(M,P)$ is the Bongartz completion of $(U,Q)$, then we have $\Fac M=\prescript{\bot}{}({\tau U})\cap Q^\bot$. Then by Proposition \ref{prominmax}, we know that $(U,Q)$ is a direct summand of $(M,P)$, say $(U,Q)=\bigoplus\limits_{i=1}^s(M_i,Q_i)$, where $s\leq n$. Thus the statement (a) holds.

Now we prove the statement (b). By Theorem \ref{thmtref}, we know that for each $1\leq r\leq n$, there exists a brick $B_r\in\mod A$ such that the $r$th column vector ${\bf c}_r$ of $C_{(M,P)}$ is given by $[B_r]$ or $-[B_r]$, where $[B_r]$ denotes the dimensional vector of $B_r$.
Assume by contradiction that there exists some $s+1\leq k\leq n$ such that ${\bf c}_k\notin \mathbb N^n$, then we must have ${\bf c}_k=-[B_k]$. Let $(M^\prime,P^\prime)$ be the mutation of $(M,P)$ in direction $k$. Since ${\bf c}_k=-[B_k]$ is a non-positive vector and by Lemma \ref{lemright}, we know that
$\prescript{\bot}{}({\tau U})\cap Q^\bot=\Fac M\subsetneqq\Fac M^\prime$. On the other hand, by $(U,Q)=\bigoplus\limits_{i=1}^s(M_i,Q_i)$ and $k\geq s+1$, we know that $(U,Q)$ is also a direct summand of $(M^\prime,P^\prime)$.  Then by Proposition \ref{prominmax}, we get $\Fac U\subseteq\Fac M^\prime\subseteq \prescript{\bot}{}({\tau U})\cap Q^\bot$. This contradicts $\prescript{\bot}{}({\tau U})\cap Q^\bot=\Fac M\subsetneqq\Fac M^\prime$. Thus the $c$-vectors in the $C$-matrix $C_{(M,P)}$ indexed by $s+1,\cdots,n$ are non-negative vectors, that is, the statement (b) holds.

$\Longleftarrow$:
By (a) and Proposition \ref{prominmax}, we get $\Fac U\subseteq\Fac M\subseteq \prescript{\bot}{}({\tau U})\cap Q^\bot$.

Now assume by contradiction that $(M,P)$ is not the Bongartz completion of $U$, that is, $\Fac M\neq \prescript{\bot}{}({\tau U})\cap Q^\bot$. Then we have $\Fac M\subsetneqq\prescript{\bot}{}({\tau U})\cap Q^\bot$. By Lemma \ref{lemair}, there exists a right mutation $(M^\prime,P^\prime)$ of $(M,P)$ such that $$\Fac U\subseteq \Fac M\subsetneqq \Fac M^\prime\subseteq \prescript{\bot}{}({\tau U})\cap Q^\bot.$$  By Proposition \ref{prominmax}, we know that $(U,Q)=\bigoplus\limits_{i=1}^s(M_i,Q_i)$ is a direct summand of $(M^\prime,P^\prime)$.

Suppose the mutation from $(M,P)$ to $(M^\prime,P^\prime)$ is in direction $p$. Because $(U,Q)=\bigoplus\limits_{i=1}^s(M_i,Q_i)$ is a common direct summand of $(M,P)$ and $(M^\prime,P^\prime)$, we know that  $p\notin\{1,\cdots,s\}$. Thus $p\geq s+1$.  Since $(M^\prime,P^\prime)$ is a right mutation of $(M,P)$ in direction $p$ and by Lemma \ref{lemright}, we know that the $p$th column vector of $C_{(M,P)}$ is a non-positive vector. Thus we find a $p\geq s+1$ such that the $p$th column vector of $C_{(M,P)}$ is a non-positive vector, which contradicts (b).
Therefore, $\Fac M=\prescript{\bot}{}({\tau U})\cap Q^\bot$ and $(M,P)$ is the Bongartz completion of $(U,Q)$.
\end{proof}

By the dual arguments, we can show the following result.
\begin{theorem} \label{thmco-completion}
Let $(U,Q)$ be a basic $\tau$-rigid pair and $(M,P)=\bigoplus\limits_{i=1}^n(M_i,Q_i)$ a basic $\tau$-tilting pair in $\mod A$. Then $(M,P)$ is the Bongartz co-completion of $(U,Q)$ if and only if the following two conditions hold.
\begin{itemize}
\item[(a)] $(U,Q)$ is a direct summand of $(M,P)$, say $(U,Q)=\bigoplus\limits_{i=1}^s(M_i,Q_i)$, where $s\leq n$;
\item[(b)] The $c$-vectors in the $C$-matrix $C_{(M,P)}$ indexed by $s+1,\cdots,n$ are non-positive vectors.
\end{itemize}
\end{theorem}

\begin{example}
Let $A$ be be a finite-dimensional $K$-algebra given by the quiver
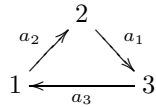
\[\xymatrixrowsep{5mm}
\xymatrixcolsep{5mm}
\xymatrix{&2\ar[rd]^{a_1}&\\
1\ar[ru]^{a_2}&&3\ar[ll]^{a_3}}\]
with relations $a_3a_1=0, a_2a_3=0$ and $a_1a_2=0$.
This is a cluster-tilted algebra \cite{bmr07} of type $A_3$ and there are $14$ basic $\tau$-tilting pairs up to isomorphisms in $\mod A$, cf. \cite{air}*{Example 6.3}. 

By Theorem \ref{ginjection},  we can represent a $\tau$-tilting pair using its $G$-matrix. The $G$-matrices of the $14$ basic $\tau$-tilting pairs are given in the following graph. (Some edges of the graph are labeled by $\mu_1,\mu_2,\mu_3$ to indicate how do we reach a given $\tau$-tilting pair $(M,P)$ from the initial $\tau$-tilting pair $(A,0)$ when we calculate the $G$-matrix of $(M,P)$ using Keller's ``quiver mutation in Java" for the opposite quiver of the quiver of the algebra $A$.)

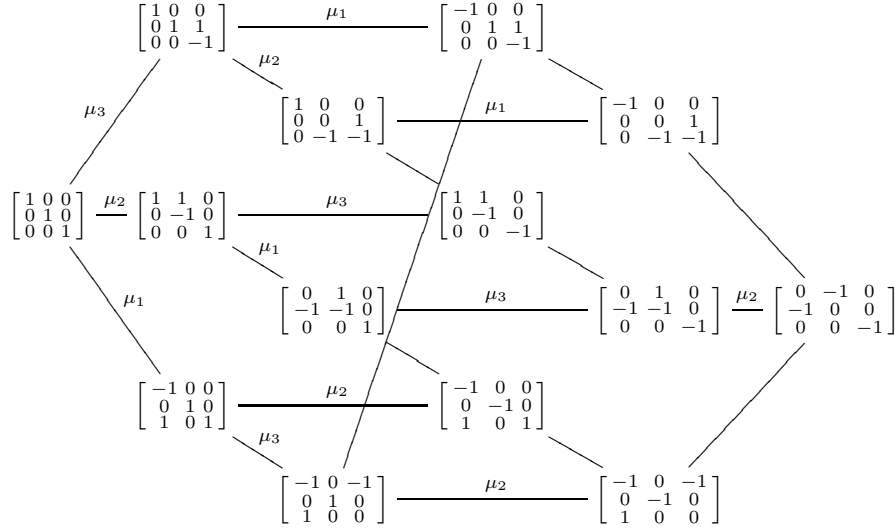
\begin{equation*} \xymatrixrowsep{4mm}
\xymatrixcolsep{4mm}
    \xymatrix { 
    &{\left[\begin{smallmatrix}
        1 & 0 & 0\\ 0 & 1& 1\\ 0& 0& -1
    \end{smallmatrix}\right]}\ar@{-}[rr]^{\mu_1}\ar@{-}[rd]^{\mu_2} & & {\left[\begin{smallmatrix}
        -1 & 0 & 0\\ 0 & 1& 1\\ 0& 0& -1
    \end{smallmatrix}\right]}\ar@{-}[rd]& &\\
    &&{\left[\begin{smallmatrix}
        1 & 0 & 0\\ 0 & 0 & 1\\ 0& -1& -1
    \end{smallmatrix}\right]}\ar@{-}[rr]^{\mu_1}\ar@{-}[rd] && {\left[\begin{smallmatrix}
        -1 & 0 & 0\\ 0 & 0& 1\\ 0& -1& -1
    \end{smallmatrix}\right]}\ar@{-}[rdd]&\\
   {\left[\begin{smallmatrix}
        1 & 0 & 0\\ 0 & 1& 0\\ 0& 0& 1
    \end{smallmatrix}\right]}\ar@{-}[r]^{\mu_2}\ar@{-}[ruu]^{\mu_3}\ar@{-}[rdd]^{\mu_1} &{\left[\begin{smallmatrix}
        1 & 1& 0\\ 0 & -1& 0\\ 0& 0& 1
    \end{smallmatrix}\right]}\ar@{-}[rr]^{\mu_3}\ar@{-}[rd]^{\mu_1}& & {\left[\begin{smallmatrix}
        1 & 1 & 0\\ 0 & -1 & 0\\ 0& 0& -1
    \end{smallmatrix}\right]}\ar@{-}[rd]& &\\
     && {\left[\begin{smallmatrix}
        0 & 1 & 0\\ -1 & -1& 0\\ 0& 0& 1
    \end{smallmatrix}\right]}\ar@{-}[rr]^{\mu_3}\ar@{-}[rd]&& {\left[\begin{smallmatrix}
        0 & 1 & 0\\ -1 & -1& 0\\ 0& 0& -1
    \end{smallmatrix}\right]}\ar@{-}[r]^{\mu_2}&{\left[\begin{smallmatrix}
        0 & -1 & 0\\ -1 & 0& 0\\ 0& 0& -1
    \end{smallmatrix}\right]}\\
     &{\left[\begin{smallmatrix}
        -1 & 0 & 0\\ 0 & 1& 0\\ 1& 0& 1
    \end{smallmatrix}\right]}\ar@{-}[rr]^{\mu_2}\ar@{-}[rd]^{\mu_3}& &{\left[\begin{smallmatrix}
        -1 & 0 & 0\\ 0 & -1& 0\\ 1& 0& 1
    \end{smallmatrix}\right]} \ar@{-}[rd]& &\\
    &&{\left[\begin{smallmatrix}
        -1 & 0 & -1\\ 0 & 1& 0\\ 1& 0& 0
    \end{smallmatrix}\right]}\ar@{-}[rr]^{\mu_2}\ar@{-}[ruuuuu] && {\left[\begin{smallmatrix}
        -1 & 0 & -1\\ 0 & -1& 0\\ 1& 0& 0
    \end{smallmatrix}\right]}\ar@{-}[ruu]&\\
    }
\end{equation*}


Consider the indecomposable basic $\tau$-rigid pair $(U,Q)$ whose $g$-vector is  $${\bf g}(U,Q)=\begin{bmatrix}0\\1\\-1\end{bmatrix}.$$
From the above graph, we can see there are $4$ $G$-matrices containing ${\bf g}(U,Q)$ as a column vector and they are
\begin{eqnarray}
G_1&=&\begin{bmatrix}1&0&0\\0&1&1\\0&0&-1\end{bmatrix},\;\;G_2=\begin{bmatrix}-1&0&0\\0&1&1\\0&0&-1\end{bmatrix},\nonumber\\
G_3&=&\begin{bmatrix}1&0&0\\0&0&1\\0&-1&-1\end{bmatrix},\;\;G_4=\begin{bmatrix}-1&0&0\\0&0&1\\0&-1&-1\end{bmatrix}.\nonumber
\end{eqnarray}
The corresponding $C$-matrices $C_i=(G_i^{\rm T})^{-1}$ are 
\begin{eqnarray}
C_1&=&\begin{bmatrix}1&0&0\\0&1&0\\0&1&-1\end{bmatrix},\;\;C_2=\begin{bmatrix}-1&0&0\\0&1&0\\0&1&-1\end{bmatrix},\nonumber\\
C_3&=&\begin{bmatrix}1&0&0\\0&-1&1\\0&-1&0\end{bmatrix},\;\;C_4=\begin{bmatrix}-1&0&0\\0&-1&1\\0&-1&0\end{bmatrix}.\nonumber\end{eqnarray}
Since the signs of the columns of $C_1,C_2,C_3,C_4$ not indexed by $(U,Q)$ are $(+,+),(-,+),(+,-),(-,-)$ respectively, we know that the Bongartz completion of $(U,Q)$ is the basic $\tau$-tilting pair whose $G$-matrix is $G_1$ and the Bongartz co-completion of $(U,Q)$ is the basic $\tau$-tilting pair whose $G$-matrix is $G_4$.
\end{example}

  \begin{remark}\label{rmkop} Notice that
  \cite{air}*{Theorem 2.14, Proposition 2.27} tell us there exists a bijection $(-)^\dag$ from the basic $\tau$-tilting pairs in $\mod A$ to those in $\mod A^{\rm op}$, which reverses the partial order given by
$$(M,P)\leq (M^\prime,P^\prime)\Longleftrightarrow \Fac M\subseteq\Fac M^\prime.$$
Thus the Bongartz co-completion in $\mod A$ can be studied via the Bongartz completion in $\mod A^{\rm op}$.
\end{remark}

\subsection{Bongartz completion in cluster algebras via $c$-vectors}
Inspired by Theorems \ref{thm-tau} and \ref{thmco-completion}, we give the definition of Bongartz completion and Bongartz co-completion in cluster algebras using $c$-vectors.
\begin{definition}[Bongartz completion and Bongartz co-completion]
\label{defcompletion}
Let $\Acal(B, t_0)$ be a cluster algebra with initial exchange matrix $B$ at $t_0$ and $U$ a subset of some cluster of $\Acal(B, t_0)$.
\begin{itemize}
\item[(i)] A cluster $[{\bf x}_s]$ is called the {\em Bongartz completion} of $U$ with respect to seed $\Sigma_{t_0}$ if the following two conditions hold.
\begin{itemize}
 \item [(a)] $U$ is a subset of $[\xx_s]$;
 \item[(b)] The $i$th column of the $C$-matrix $C_s^{B;t_0}$ is a non-negative vector for any $i$ such that $x_{i;s}\notin U$.
\end{itemize}
\item[(ii)] A cluster $[{\bf x}_s]$ is called the {\em Bongartz co-completion} of $U$ with respect to seed $\Sigma_{t_0}$ if the following two conditions hold.
\begin{itemize}
 \item [(a$'$)] $U$ is a subset of $[\xx_s]$;
 \item[(b$'$)] The $i$th column of the $C$-matrix $C_s^{B;t_0}$ is a non-positive vector for any $i$ such that $x_{i;s}\notin U$.
\end{itemize}
\end{itemize}
\end{definition}

\begin{remark}
Now we compare the Bongartz co-completion in this paper with that in \cite{cao}.
\begin{itemize}
    \item The input of Bongartz co-completion in this paper is an initial seed $\Sigma_{t_0}$ and a subset $U$ of some cluster and the output is a final cluster $[{\bf x}_s]$ such that 
 $U\subseteq [{\bf x}_s]$ and the columns of the $C$-matrix $C_s^{B_{t_0};t_0}$ satisfies some non-positive condition (see Definition \ref{defcompletion} (ii)).
 \item The input of Bongartz co-completion in \cite{cao} (i.e., $g$-pairs in \cite{cl2}) is a final seed $\Sigma_t$ and a subset $U$ of some cluster and the output is an initial cluster $[{\bf x}_{t^\prime}]$ such that 
$U\subseteq [{\bf x}_{t^\prime}]$ and the rows of the $G$-matrix $G_t^{B_{t^\prime};t^\prime}$ satisfies some non-negative condition (see Theorem \ref{thmgcompletion}).
\item The Bongartz completion in this paper for cluster algebra $\mathcal A(B,t_0)$ is closely related to the Bongartz co-completion in \cite{cao} (i.e., $g$-pairs in \cite{cl2}) for cluster algebra $\mathcal A(B^{\rm T},t_0)$ via tropical duality in Lemma \ref{lemcg}. We will see it in the proof of Theorem \ref{bongartz-completion}.
\end{itemize}
\end{remark}

Notice that both the existence and uniqueness of Bongartz completion and Bongartz co-completion in cluster algebras are not clear from their own definitions. Before the discussion on their existence and uniqueness, we first discuss the relationship between Bongartz completion and Bongartz co-completion.

Let $\Sigma_{t_0}=({\bf x}_{t_0},{\bf y}_{t_0},B_{t_0})$ be a seed over $\mathbb P$. Put $${\bf y}_{t_0}^{-1}:=(y_{1;t_0}^{-1},\cdots,y_{n;t_0}^{-1}).$$
We call the new seed $\Sigma_{t_0}^{\rm op}:=({\bf x}_{t_0},{\bf y}_{t_0}^{-1}, -B_{t_0})$  the {\em opposite seed} of $\Sigma_{t_0}$.
It can be checked that $(\mu_k(\Sigma_{t_0}))^{\rm op}=\mu_k(\Sigma_{t_0}^{\rm op})$. Thus the cluster algebra $\mathcal A(\Sigma_{t_0})$ given by the seed $\Sigma_{t_0}$ and the cluster algebra $\mathcal A(\Sigma_{t_0}^{\rm op})$ given by the opposite seed $\Sigma_{t_0}^{\rm op}$ have the same set of cluster variables and the same set of clusters.

\begin{proposition}\label{prop:BcBcc}
Let $\Acal(\Sigma_{t_0})$ be a cluster algebra with initial seed $\Sigma_{t_0}$ and $\mathcal A(\Sigma_{t_0}^{\rm op})$ the cluster algebra with initial seed $\Sigma_{t_0}^{\rm op}$. Let $U$ be a subset of some cluster of $\Acal(\Sigma_{t_0})$ and suppose there exists a seed $\Sigma_{t_1}=\overset{\leftarrow}{\mu}(\Sigma_{t_0})$ of $\Acal(\Sigma_{t_0})$ such that
 $$\begin{bmatrix}B_{t_1}\\-I_n\end{bmatrix}=\overset{\leftarrow}{\mu}\begin{bmatrix}B_{t_0}\\ I_n\end{bmatrix}.$$
 Then a cluster $[{\bf x}_s]$ is the Bongartz completion of $U$ with respect to $\Sigma_{t_0}$ in $\Acal(\Sigma_{t_0})$ if and only if
 it is the Bongartz co-completion of $U$ with respect to $\Sigma_{t_1}^{\rm op}$ in $\mathcal A(\Sigma_{t_0}^{\rm op})$.
\end{proposition}
\begin{proof}
Let $\overset{\leftarrow}{\nu}$ be a sequence of mutations such that $\Sigma_s=\overset{\leftarrow}{\nu}(\Sigma_{t_0})$. Then we have
$$\begin{bmatrix}B_s\\ C_{s}^{B_{t_0};t_0} \end{bmatrix}=\overset{\leftarrow}{\nu}\begin{bmatrix}B_{t_0}\\ I_n\end{bmatrix}.$$
By $\begin{bmatrix}B_{t_1}\\-I_n\end{bmatrix}=\overset{\leftarrow}{\mu}\begin{bmatrix}B_{t_0}\\ I_n\end{bmatrix}$, we get
$\begin{bmatrix}B_s\\ C_{s}^{B_{t_0};t_0} \end{bmatrix}=\overset{\leftarrow}{\nu}\begin{bmatrix}B_{t_0}\\ I_n\end{bmatrix}
=\overset{\leftarrow}{\nu}(\overset{\leftarrow}{\mu})^{-1}\begin{bmatrix}B_{t_1}\\ -I_n\end{bmatrix}$.
Then by the definition of mutations, we have

$$\begin{bmatrix}-B_s\\ -C_{s}^{B_{t_0};t_0} \end{bmatrix}
=-\overset{\leftarrow}{\nu}(\overset{\leftarrow}{\mu})^{-1}\begin{bmatrix}B_{t_1}\\ -I_n\end{bmatrix}=
\overset{\leftarrow}{\nu}(\overset{\leftarrow}{\mu})^{-1}\begin{bmatrix}-B_{t_1}\\ I_n\end{bmatrix}.$$
Hence, $C_s^{-B_{t_1};t_1}=-C_s^{B_{t_0};t_0}$, where $C_s^{-B_{t_1};t_1}$ is the $C$-matrix of the seed $\Sigma_{s}^{\rm op}$ with respect to $\Sigma_{t_1}^{\rm op}$.
Then the result follows from the definition of Bongartz completion and Bongartz co-completion.
\end{proof}
\begin{remark}
The seed $\Sigma_{t_1}$ in Proposition \ref{prop:BcBcc} is analogous to the $\tau$-tilting pair $(0,A)$ for a finite dimensional basic algebra $A$. Then one can see that Proposition \ref{prop:BcBcc} is analogous to Remark \ref{rmkop}.
\end{remark}

In the sequel, we will prove that Bongartz completion in cluster algebras always exists and it is unique. However, this is not the case for Bongartz co-completion. The existence and uniqueness of Bongartz co-completion in cluster algebras will be discussed in Corollary \ref{corunique}.

\begin{lemma}\label{lembijection}
Let $\mathcal A(B,t_0)$ be a cluster algebra with initial exchange matrix $B$ at $t_0$, and $\mathcal A(B^{\rm T}, t_0)$ be a cluster algebra with initial exchange matrix $B^{\rm T}$ at $t_0$. We denote by $\Sigma_t=({\bf x}_t,{\bf y}_t, B_t)$ the seed of $\mathcal A(B, t_0)$ at $t$ and by $\Sigma^{\rm T}=({\bf x}_t^{\rm T}, {\bf y}_t^{\rm T}, B_t^{\rm T})$ the seed of $\mathcal A(B^{\rm T}, t_0)$ at $t$. Then $x_{i;t} =x_{j;t'}$ if and only if $x_{i;t}^{\rm T}=x_{j;t^\prime}^{\rm T}$, where $t, t'\in\TT_n$ and $i, j\in \{1,2,\cdots, n\}$. In particular, $(-)^{\rm T}$ induces a bijection from the cluster variables of $\mathcal A(B, t_0)$ to those of $\mathcal A(B^{\rm T}, t_0)$.
\end{lemma}
\begin{proof}
$\Longrightarrow$: Suppose $x_{i;t}=x_{j;t^\prime}$, then by Theorem \ref{thmgraph}, there exists a sequence $\mu_{k_\ell}\cdots\mu_{k_1}$ of mutations with $i\neq k_1,\cdots,k_\ell$ such that the two seeds
$\Sigma_{t^\prime}$ and $$\Sigma_s:=\mu_{k_\ell}\cdots\mu_{k_1}(\Sigma_t)$$
are equivalent via a permutation $\sigma$. By \cite{nak20}*{Theorem 5.20}, the two seeds $\Sigma_{t^\prime}^{\rm T}$ and
$$\Sigma_s^{\rm T}:=\mu_{k_\ell}\cdots\mu_{k_1}(\Sigma_t^{\rm T})$$
are also equivalent via the same permutation $\sigma$. Thus we have $x_{j;t^\prime}=x_{\sigma(j);s}$ and $x_{j;t^\prime}^{\rm T}=x_{\sigma(j);s}^{\rm T}$.

By $i\neq k_1,\cdots,k_\ell$, we have $x_{i;t}=x_{i;s}$ and $x_{i;t}^{\rm T}=x_{i;s}^{\rm T}$. Since the equality $x_{i;s}=x_{i;t}=x_{j;t^\prime}=x_{\sigma(j);s}$ holds, we have $i=\sigma(j)$. Thus
$x_{i;t}^{\rm T}=x_{i;s}^{\rm T}=x_{\sigma(j);s}^{\rm T}=x_{j;t^\prime}^{\rm T}$.

$\Longleftarrow:$ Similarly, we can show that if $x_{i;t}^{\rm T}=x_{j;t^\prime}^{\rm T}$, then $x_{i;t}=x_{j;t^\prime}$.
\end{proof}

\begin{lemma}[{\cite{nz}*{(1.13)}}]\label{lemcg} In a cluster algebra, for any $t_0$,$t\in\TT_n$, we have
$$
    (C_t^{B_{t_0};t_0})^{\rm T}=G_{t_0}^{B_t^{\rm T};t}.
$$
\end{lemma}

\begin{theorem}\label{bongartz-completion}
Let $\mathcal A(B, t_0)$ be a cluster algebra with initial exchange matrix $B$ at $t_0$. Then for any subset $U$ of a cluster $[{\bf x}_t]$, there exists a unique cluster $[{\bf x}_s]$ such that $[{\bf x}_s]$ is the Bongartz completion of $U$ with respect to $t_0$.
\end{theorem}

\begin{proof}
Let $\mathcal A(B^{\rm T}, t_0)$ be a cluster algebra with initial exchange matrix $B^{\rm T}$ at $t_0$. Let $U^{\rm T}$ be the image of $U$ under the bijection in Lemma \ref{lembijection}, which is a subset of $[{\bf x}_t^\mathrm{T}]$, by $U\subseteq[{\bf x}_t]$.

Claim:  $[{\bf x}_s]$ is the Bongartz completion of $U$ with respect to $\Sigma_{t_0}$ in $\mathcal A(B, t_0)$ if and only if $([{\bf x}_{t_0}^{\rm T}],[{\bf x}_s^{\rm T}])$ is a $g$-pair associated with $U^{\rm T}$ in $\mathcal A(B^{\rm T}, t_0)$.

This claim follows from Lemmas \ref{lembijection} and \ref{lemcg}, which transfer the two conditions (a) and (b) in Definition \ref{defcompletion} for $\mathcal A(B, t_0)$ to the two conditions (a) and (b) in Theorem \ref{thmgcompletion} for $\mathcal A(B^{\rm T}, t_0)$.

By applying Theorem \ref{thmgpair} to cluster algebra $\mathcal A(B^{\rm T}, t_0)$, we get the existence and uniqueness of the Bongartz completions in $\mathcal A(B, t_0)$.
\end{proof}

We denote by $B_U(\Sigma_{t_0})$ the Bongartz completion of $U$ with respect to seed $\Sigma_{t_0}$.

\begin{proposition}\label{prott}
Let $\mathcal A$ be a cluster algebra and $U$ a subset of some cluster of $\mathcal A$. Suppose that $\Sigma_t$ and $\Sigma_{t^\prime}$ are two equivalent seeds of $\mathcal A$, then
$B_U(\Sigma_t)=B_U(\Sigma_{t^\prime})$. In particular, the Bongartz completion of $U$ with respect to $\Sigma_t$ only depends on $U$ and the equivalent class of the seed $\Sigma_t$.
\end{proposition}
\begin{proof}
The assertion follows from the definition of Bongartz completion.
\end{proof}

Thanks to Theorem \ref{non-labeled-cluster-thm} and Proposition \ref{prott}, we know that in a cluster algebra the Bongartz completion $B_U(\Sigma_t)$ of $U$ with respect to seed $\Sigma_t$ only depends on $U$ and $[{\bf x}_t]$. Thus we just denote $B_U[{\bf x}_t]=B_U(\Sigma_t)$.

\subsection{Commutativity of Bongartz completions}
In this subsection, we prove that Bongartz completion admits certain commutativity.
\begin{lemma}\label{lempositive}
Let $B$ be an $n\times n$ skew-symmetrizable matrix and $P$ an $m\times n$ non-negative matrix. Then for any sequence of mutations $\overset{\leftarrow}{\mu}$, we have

\begin{eqnarray}
\overset{\leftarrow}{\mu}\left(\begin{bmatrix}I_n&0\\0&P\end{bmatrix}\begin{bmatrix}B\\I_n\end{bmatrix}\right)=\overset{\leftarrow}{\mu}\left(\begin{bmatrix}B\\ P\end{bmatrix}\right)
=\begin{bmatrix}I_n&0\\0&P\end{bmatrix}\overset{\leftarrow}{\mu}\left(\begin{bmatrix}B\\ I_n\end{bmatrix}\right).\nonumber
\end{eqnarray}
\end{lemma}
\begin{proof}
The assertion follows from \cite{cl2019}*{Lemma 3.1} and the sign-coherence of $C$-matrices.
\end{proof}

\begin{theorem}\label{thmcommu}
Let $\mathcal A$ be a cluster algebra with initial seed $\Sigma_{t_0}$ and $U=W\sqcup V$ a subset of some cluster of $\mathcal A$. Then we have
$B_U[{\bf x}_{t_0}]=B_UB_W[{\bf x}_{t_0}]=B_VB_W[{\bf x}_{t_0}]$.
\end{theorem}
\begin{proof}
Let $\Sigma_{s_1}$ and $\Sigma_{s_2}$ be two seeds of $\mathcal A$ such that $[{\bf x}_{s_1}]=B_W[{\bf x}_{t_0}]$ and $[{\bf x}_{s_2}]=B_U[{\bf x}_{s_1}]$. Then we have
$$B_UB_W[{\bf x}_{t_0}]=B_U[{\bf x}_{s_1}]=[{\bf x}_{s_2}]\;\;\;\text{and}\;\;\;B_VB_W[{\bf x}_{t_0}]=B_V[{\bf x}_{s_1}].$$
Thus it suffices to prove (i) $[{\bf x}_{s_2}]=B_V[{\bf x}_{s_1}]$ and (ii) $[{\bf x}_{s_2}]=B_U[{\bf x}_{t_0}]$.

Before proving (i) and (ii), we first introduce a seed $\Sigma_{s_3}$, which will be used in the proofs of (i) and (ii). By $[{\bf x}_{s_1}]=B_W[{\bf x}_{t_0}]$, we have $W\subseteq [{\bf x}_{s_1}]$. Without loss of generality, we can assume $W=\{x_{p+1;s_1},\cdots,x_{n;s_1}\}$.
By $[{\bf x}_{s_2}]=B_U[{\bf x}_{s_1}]$, we know that $U=W\sqcup V\subseteq [{\bf x}_{s_2}]$. Thus $W$ is a common subset of $[{\bf x}_{s_1}]$ and $[{\bf x}_{s_2}]$.
 Then by Theorem \ref{thmgraph}, there exists a sequence of mutations $\overset{\leftarrow}{\mu}$ such that
 the seed $\Sigma_{s_3}:=\overset{\leftarrow}{\mu}(\Sigma_{s_1})$ is equivalent to the seed $\Sigma_{s_2}$ and the common subset $W=\{x_{p+1;s_1},\cdots,x_{n;s_1}\}$ keeps unchanged during this process.
 This means that all the mutations in $\overset{\leftarrow}{\mu}$ are in directions $1,\cdots,p$.
Thus the lower part $C_{s_3}^{s_1}$ of $\begin{bmatrix}B_{s_3}\\ C_{s_3}^{s_1}\end{bmatrix}=\overset{\leftarrow}{\mu}\begin{bmatrix}B_{s_1}\\ I_n\end{bmatrix}$ has the following form $$C_{s_3}^{s_1}=\begin{bmatrix}M&N\\0&I_{n-p}\end{bmatrix}.$$
By the sign-coherence of $C$-matrices, we know that $N$ is a non-negative matrix.

Now we begin the proofs of (i) and (ii).

(i) It is clear that $V$ is a subset of $[{\bf x}_{s_3}]=[{\bf x}_{s_2}]=B_U[{\bf x}_{s_1}]$. We only need to show that the columns of the $C$-matrix $C_{s_3}^{s_1}$ not indexed by $V$ are non-negative vectors.
By $[{\bf x}_{s_3}]=[{\bf x}_{s_2}]=B_U[{\bf x}_{s_1}]$, we know that the columns of $C_{s_3}^{s_1}$ not indexed by $U=W\sqcup V$ are non-negative vectors.
 The columns of $C_{s_3}^{s_1}$ indexed by $W=\{x_{p+1;s_1},\cdots,x_{n;s_1}\}$ are non-negative vectors since $\begin{bmatrix}N\\ I_{n-p}\end{bmatrix}$ is a non-negative matrix. Thus the columns of the $C$-matrix $C_{s_3}^{s_1}$ not indexed by $V$ are non-negative vectors. Hence, $B_V[{\bf x}_{s_1}]=[{\bf x}_{s_3}]=[{\bf x}_{s_2}]$.

(ii) By $[{\bf x}_{s_3}]=[{\bf x}_{s_2}]=B_U[{\bf x}_{s_1}]$, we have $U\subseteq [{\bf x}_{s_3}]$. We only need to show that the columns of the $C$-matrix $C_{s_3}^{t_0}$ not indexed by $U=W\sqcup V$ are non-negative vectors.

 Let ${\overset{\leftarrow}{\nu}}$ be a sequences of mutations such that $\Sigma_{s_1}={\overset{\leftarrow}{\nu}}(\Sigma_{t_0})$. We rewrite $\begin{bmatrix}B_{s_1}\\ C_{s_1}^{t_0}\end{bmatrix}=\overset{\leftarrow}{\nu}\begin{bmatrix}B_{t_0}\\ I_n\end{bmatrix}$
 as the following block matrix:
\begin{eqnarray}\label{eqnnu}
\begin{bmatrix}B_{s_1}\\ C_{s_1}^{t_0}\end{bmatrix}=\begin{bmatrix}(B_1)_{p\times p}&B_2\\B_3&(B_4)_{(n-p)\times (n-p)}\\ P_{n\times p}& Q_{n\times (n-p)}\end{bmatrix}.
\end{eqnarray}
By $[{\bf x}_{s_1}]=B_W[{\bf x}_{t_0}]$ and $W=\{x_{p+1;s_1},\cdots,x_{n;s_1}\}$, we know that $P$ is a non-negative matrix.
We rewrite $\begin{bmatrix}B_{s_3}\\ C_{s_3}^{s_1}\end{bmatrix}=\overset{\leftarrow}{\mu}\begin{bmatrix}B_{s_1}\\ I_n\end{bmatrix}$ as the following block matrix:
$$\begin{bmatrix}B_{s_3}\\ C_{s_3}^{s_1}\end{bmatrix}=\begin{bmatrix}B_1^\prime&B_2^\prime\\B_3^\prime&B_4^\prime\\ C_1^\prime&C_2^\prime\\C_3^\prime&C_4^\prime\end{bmatrix}=\overset{\leftarrow}{\mu}\begin{bmatrix}B_1&B_2\\B_3&B_4\\ I_p&0\\0&I_{n-p}\end{bmatrix}=\overset{\leftarrow}{\mu}\begin{bmatrix}B_{s_1}\\ I_n\end{bmatrix}.$$
By $[{\bf x}_{s_3}]=B_V[{\bf x}_{s_1}]$, we know that the columns of $C_{s_3}^{s_1}$ not indexed by $V$ are non-negative vectors. In particular, the columns of $C_1^\prime$ not indexed
by $V$ are non-negative vectors.

By $\Sigma_{s_3}=\overset{\leftarrow}{\mu}(\Sigma_{s_1})$, we know $\begin{bmatrix}B_{s_3}\\C_{s_3}^{t_0}\end{bmatrix}=\overset{\leftarrow}{\mu}\begin{bmatrix}B_{s_1}\\ C_{s_1}^{t_0} \end{bmatrix}$.  We rewrite it as the following block matrix:
$$\begin{bmatrix}B_{s_3}\\C_{s_3}^{t_0}\end{bmatrix}=\begin{bmatrix}B_1^\prime&B_2^\prime\\B_3^\prime&B_4^\prime\\P^\prime&Q^\prime\end{bmatrix}=\overset{\leftarrow}{\mu}\begin{bmatrix}B_1&B_2\\B_3&B_4\\ P& Q\end{bmatrix}=\overset{\leftarrow}{\mu}\begin{bmatrix}B_{s_1}\\ C_{s_1}^{t_0} \end{bmatrix}.$$

Recall that we want to show that the columns of $C_{s_3}^{t_0}$ not indexed by $U=W\sqcup V$ are non-negative vectors. Notice that the columns of $Q^\prime$ are exactly the columns of $C_{s_3}^{t_0}$ indexed by $W$. In the sequel, we just focus on the columns of $P^\prime$ not indexed by $V$.

 Since the mutations in $\overset{\leftarrow}{\mu}$ are in directions $1,\cdots,p$, we have $$\begin{bmatrix}B_1^\prime\\ C_1^\prime\end{bmatrix}=\overset{\leftarrow}{\mu}\begin{bmatrix}B_1\\ I_{p}\end{bmatrix}\;\;\;\text{and}\;\;\;\; \begin{bmatrix}B_1^\prime\\ P^\prime\end{bmatrix}=\overset{\leftarrow}{\mu}\begin{bmatrix}B_1\\ P\end{bmatrix}.$$
 Since $P$ is a non-negative matrix and by applying Lemma \ref{lempositive} to $\begin{bmatrix}B_1^\prime\\ C_1^\prime\end{bmatrix}=\overset{\leftarrow}{\mu}\begin{bmatrix}B_1\\ I_{p}\end{bmatrix}$ and
$
\begin{bmatrix}B_1^\prime\\ P^\prime\end{bmatrix}=\overset{\leftarrow}{\mu}\begin{bmatrix}B_1\\ P\end{bmatrix}
$,
we get $P^\prime=PC_1^\prime$. Because $P$ is a non-negative matrix and the columns of $C_1^\prime$ not indexed by $V$ are non-negative vectors, we get that the columns of $P^\prime$ not indexed by $V$ are non-negative vectors. Therefore, the columns of $C_{s_3}^{s_0}$ not indexed by $U=W\sqcup V$ are non-negative vectors. Hence, we get $B_U[{\bf x}_{t_0}]=[{\bf x}_{s_3}]=[{\bf x}_{s_2}]$.

Finally, by the results in (i) and (ii), we get $B_U[{\bf x}_{t_0}]=B_UB_W[{\bf x}_{t_0}]=B_VB_W[{\bf x}_{t_0}]$.
\end{proof}

It is clear that Theorem \ref{thmcommu} induces the following commutativity of Bongartz completion:
$$B_VB_W[{\bf x}_{t_0}]=B_WB_V[{\bf x}_{t_0}],$$
where $V$ and $W$ are as in Theorem \ref{thmcommu}.

\begin{remark}
Based on the commutativity of Bongartz completion, in a follow up work, we will give the explicit construction of Bongartz completion on orbifolds corresponding to cluster algebras.
\end{remark}

\section{Applications}
In this section, we give two applications of Bongartz completion in cluster algebras.
\subsection{Application to exchange quivers}

\begin{definition}[Green mutation and green-to-red sequence] Let $\mathcal A(B, t_0)$ be a cluster algebra with initial exchange matrix $B$ at $t_0$.
\begin{itemize}
\item[(i)] A seed mutation $\mu_k(\Sigma_t)$ in $\mathcal A(B, t_0)$ is called a {\em green mutation}, if the $k$th column of the $C$-matrix $C_t^{B;t_0}$ is a non-negative vector. Otherwise, it is called a {\em red mutation}.
    \item[(ii)]  Let $\overset{\leftarrow}{\mu}$ be a sequence of mutations of $B$ and denote by $\Sigma_t=\overset{\leftarrow}{\mu}(\Sigma_{t_0})$. The sequence $\overset{\leftarrow}{\mu}$ is called a {\em green-to-red sequence} of $B$, if the $C$-matrix $C_t^{B;t_0}$ is a non-positive matrix.
 \end{itemize}
\end{definition}

Thanks to Proposition \ref{procomat}, we know that green mutation is well-defined up to seed equivalence.

\begin{definition}[Exchange quiver] Let $\mathcal A(B, t_0)$ be a cluster algebra with initial exchange matrix $B$ at $t_0$.
The {\em exchange quiver} (or {\em oriented exchange graph}) of $\mathcal A(B, t_0)$ is the quiver $\overrightarrow{\Gamma}(B,t_0)$ whose vertices correspond to the seeds of $\mathcal A(B, t_0)$ up to seed equivalence and whose arrows correspond to green mutations.
\end{definition}

\begin{remark}\label{rmkquiver}
(i) By Lemma \ref{lemright}, we know that red mutations in cluster algebras are parallel to right mutations in $\tau$-tilting theory.

(ii) The exchange quiver of  a cluster algebra $\mathcal A(B,t_0)$ is parallel to the $\tau$-tilting quiver \cite{air} in $\tau$-tilting theory for a finite dimensional basic algebra $A$, which coincides with the Hasse quiver of the basic $\tau$-tilting pairs in $\mod A$ (see \cite{air}*{Corollary 2.34}).

(iii) $B$ admits a green-to-red sequence if and only if the exchange quiver of $\mathcal A(B, t_0)$ has a sink.

(iv) The $\tau$-tilting quiver of $\mod A$ has a unique source corresponding to the basic $\tau$-tilting pair $(A,0)$ and a unique sink corresponding to the basic $\tau$-tilting pair $(0,A)$. The exchange quiver of a cluster algebra $\mathcal A(B,t_0)$ has a unique source corresponding to its initial seed, however, it may have no sinks, for example, for a cluster algebra whose initial exchange matrix is given by the following Markov matrix (see \cite{pla12}*{Example 4.3} for details):
$$B=\begin{bmatrix}0&2&-2\\-2&0&2\\2&-2&0\end{bmatrix}.$$
\end{remark}

For a subset $U$ of some cluster of $\mathcal A(B, t_0)$, we denote by $\overrightarrow{\Gamma}_U(B,t_0)$ the full subquiver of $\overrightarrow{\Gamma}(B,t_0)$ whose vertices correspond to the seeds $\Sigma_t$ of $\mathcal A(B, t_0)$ satisfying $U\subseteq [{\bf x}_t]$.
In this subsection, we are devoted to study the subquiver $\overrightarrow{\Gamma}_U(B,t_0)$.

\begin{theorem}\label{thmquiver}
Let $\mathcal A=\mathcal A(B, t_0)$ be a cluster algebra with initial exchange matrix $B$ at $t_0$ and $U$ a subset of some cluster of $\mathcal A$. Let $\Sigma_s$ be a seed of $\mathcal A$ such that $[{\bf x}_s]=B_U[{\bf x}_{t_0}]$ and denote $V:=[{\bf x}_s]\setminus U$.
Then  $\overrightarrow{\Gamma}_U(B,t_0)$ is isomorphic to the exchange quiver $\overrightarrow{\Gamma}(B_s^\dag,s)$, where $B_s^\dag$ is the $V\times V$-indexed submatrix of $B_s$.
\end{theorem}

\begin{proof}
 Since the claim does not depend on the choice of coefficients of cluster algebras, we can just assume that $\mathcal A$ is a geometric cluster algebra with initial seed $\Sigma_{t_0}=({\bf x}, B)$.

By $[{\bf x}_s]=B_U[{\bf x}_{t_0}]$, we know that $U\subseteq [{\bf x}_s]$. Without loss of generality, we can assume $$U=\{x_{p+1;s},\cdots,x_{n;s}\}.$$
We write $B_s$ as a block matrix $B_s=\begin{bmatrix}B_1&B_2\\B_3&B_4\end{bmatrix}$, where $B_1$ is a $p\times p$ submatrix of $B_s$. By the definition of $B_s^\dag$, we have $B_1=B_s^\dag$.

We denote by $\Sigma_s^U:=({\bf x}_s,\begin{bmatrix}B_1\\B_3 \end{bmatrix})=({\bf x}_s,\begin{bmatrix}B_s^\dag\\B_3 \end{bmatrix})$ the seed obtained from $\Sigma_s=({\bf x}_s, B_s)$ by freezing the cluster variable in $U$. Let $\mathcal A_U$ be the geometric cluster algebra with initial seed $\Sigma_s^U$ at the vertex $s$. Then $\overrightarrow{\Gamma}(B_s^\dag,s)$ can be viewed as the exchange quiver of $\mathcal A_U$.

Clearly, there exists a map $(-)^U$ from the vertices of $\overrightarrow{\Gamma}_U(B,t_0)$ to the vertices of  $\overrightarrow{\Gamma}(B_s^\dag,s)$ induced by $\Sigma_t\mapsto \Sigma_t^U$, where $\Sigma_t^U$ is the seed obtained from $\Sigma_t$ by freezing the cluster variables in $U$. By Theorem \ref{thmgraph}, we know that the map $(-)^U$ induces an isomorphism from $\overrightarrow{\Gamma}_U(B,t_0)$ to  $\overrightarrow{\Gamma}(B_s^\dag,s)$ as unoriented graphs.

It remains to show the map $(-)^U$ preserves the orientations, that is, we need to show that $\Sigma_r=\mu_k(\Sigma_t)$ is a green mutation in $\overrightarrow{\Gamma}_U(B,t_0)$
if and only if $\Sigma_r^U=\mu_k(\Sigma_t^U)$ is a green mutation in $\overrightarrow{\Gamma}(B_s^\dag,s)$. Notice that the seeds here are considered up to seed equivalence.

By $[{\bf x}_s]=B_U[{\bf x}_{t_0}]$ and $U=\{x_{p+1;s},\cdots,x_{n;s}\}$, we know that the $n\times p$ matrix $P$ appearing in the following block matrix is a non-negative matrix:
\begin{eqnarray}\label{eqnqp}
\begin{bmatrix}B_s\\C_s^{B;t_0}\end{bmatrix}=\begin{bmatrix}B_1&B_2\\B_3&B_4\\P&Q\end{bmatrix}.\nonumber
\end{eqnarray}
Let $\overleftarrow{\mu}$ be a sequence of mutations such that $\Sigma_t=\overleftarrow{\mu}(\Sigma_s)$. By Theorem \ref{thmgraph} and $$U=\{x_{p+1;s},\cdots,x_{n;s}\}\subseteq [{\bf x}_t]\cap[{\bf x}_s],$$ we can assume that all the mutations in $\overleftarrow{\mu}$ are in directions $1,\cdots,p$.

By $\Sigma_t=\overleftarrow{\mu}(\Sigma_s)$, we have $\begin{bmatrix}B_t\\C_t^{B;t_0}\end{bmatrix}=\overleftarrow{\mu}\begin{bmatrix}
B_s\\ C_s^{B;t_0}\end{bmatrix}$. We rewrite it as the following block matrix:
$$\begin{bmatrix}B_t\\C_t^{B;t_0}\end{bmatrix}=\begin{bmatrix}B_1^\prime&B_2^\prime\\B_3^\prime&B_4^\prime\\P^\prime&Q^\prime \end{bmatrix}=\overleftarrow{\mu}\begin{bmatrix}B_1&B_2\\B_3&B_4\\P&Q\end{bmatrix}=\overleftarrow{\mu}\begin{bmatrix}
B_s\\ C_s^{B;t_0}\end{bmatrix}.$$
Since the mutations in $\overleftarrow{\mu}$ are in directions $1,\cdots,p$, we have $\begin{bmatrix}B_1^\prime\\P^\prime \end{bmatrix}=\overleftarrow{\mu}\begin{bmatrix}B_1\\P \end{bmatrix}$. By $\Sigma_t^U=\overleftarrow{\mu}(\Sigma_s^U)$, we have $\begin{bmatrix}B_1^\prime\\C_t^{B_1;s}\end{bmatrix}=\overleftarrow{\mu}\begin{bmatrix}
B_1\\ I_p\end{bmatrix}$.

Since $P$ is a non-negative matrix and by applying
 Lemma \ref{lempositive} to $\begin{bmatrix}B_1^\prime\\P^\prime \end{bmatrix}=\overleftarrow{\mu}\begin{bmatrix}B_1\\P \end{bmatrix}$ and $\begin{bmatrix}B_1^\prime\\C_t^{B_1;s}\end{bmatrix}=\overleftarrow{\mu}\begin{bmatrix}
B_1\\ I_p\end{bmatrix}$, we get $P^\prime=PC_t^{B_1;s}$. Then by the sign-coherence of $C$-matrices, we know that the $k$th column of $P^\prime$ is a non-negative vector if and only if the $k$th column of $C_t^{B_1;s}$ is a non-negative vector.
This means that $\Sigma_r=\mu_k(\Sigma_t)$ is a green mutation in $\overrightarrow{\Gamma}_U(B,t_0)$
if and only if $\Sigma_r^U=\mu_k(\Sigma_t^U)$ is a green mutation in $\overrightarrow{\Gamma}(B_1,s)=\overrightarrow{\Gamma}(B_s^\dag,s)$. Thus the map $(-)^U$ from $\overrightarrow{\Gamma}_U(B,t_0)$ to $\overrightarrow{\Gamma}(B_s^\dag,s)$ preserves the orientations. Hence, the subquiver $\overrightarrow{\Gamma}_U(B,t_0)$ of $\overrightarrow{\Gamma}(B,t_0)$ is isomorphic to the  exchange quiver $\overrightarrow{\Gamma}(B_s^\dag,s)$ of the cluster algebra $\mathcal A_U$ with initial seed $\Sigma_s^U=({\bf x}_s,\begin{bmatrix}B_s^\dag\\B_3 \end{bmatrix})$ at the vertex $s$.
\end{proof}

We remark that Theorem \ref{thmquiver} is parallel to Jasso's $\tau$-tilting reduction \cite{jasso}*{Theorem 1.1} in $\tau$-tilting theory.

\begin{corollary}\label{corunique}
Keep the notations in Theorem \ref{thmquiver}. The following statements hold.
\begin{itemize}
\item[(i)] A vertex $\Sigma_t$ in $\overrightarrow{\Gamma}_U(B,t_0)\cong \overrightarrow{\Gamma}(B_s^\dag,s)$ corresponds to the Bongartz completion of $U$ with respect to $\Sigma_{t_0}$ if and only if $\Sigma_t$ is a source in  $\overrightarrow{\Gamma}_U(B,t_0)\cong \overrightarrow{\Gamma}(B_s^\dag,s)$.
    \item[(ii)] A vertex $\Sigma_t$ in $\overrightarrow{\Gamma}_U(B,t_0)\cong \overrightarrow{\Gamma}(B_s^\dag,s)$ corresponds to the Bongartz co-completion of $U$ with respect to $\Sigma_{t_0}$ if and only if $\Sigma_t$ is a sink in  $\overrightarrow{\Gamma}_U(B,t_0)\cong \overrightarrow{\Gamma}(B_s^\dag,s)$.
        \item [(iii)] If the Bongartz co-completion of $U$ with respect to $\Sigma_{t_0}$ exists, then it is unique.
\end{itemize}
\end{corollary}
\begin{proof}
(i) and (ii) follow from the definitions of exchange quiver, Bongartz completion and Bongartz co-completion.

(iii) Let $\mathcal A_U$ be a cluster algebra with initial exchange matrix $B_s^\dag$ at vertex $s$. 
By (ii) and the assumption that the Bongartz co-completion of $U$ with respect to $\Sigma_{t_0}$ exists, we know that the exchange quiver $\overrightarrow{\Gamma}(B_s^\dag,s)$ of $\mathcal A_U$ has a sink. Namely, there exists a seed $\Sigma_t^U$ of  $\mathcal A_U$ 
such that its $C$-matrix is a non-positive matrix.  Such a seed $\Sigma_t^U$ is unique up to seed equivalence, by \cite{qin}*{Section 2.3}.
\end{proof}

\begin{proposition}[\cite{mur}*{Corollary 19}, \cite{qin17}*{Proposition 5.1.4}] \label{promu}
Let $\mathcal A=\mathcal A(B, t_0)$ be a cluster algebra with initial exchange matrix $B$ at $t_0$. If $B$ admits a green-to-red sequence, then so does any exchange matrix of $\mathcal A$.
\end{proposition}
\begin{remark}
The above result is proved by Muller \cite{mur} in skew-symmetric case by considering the mutation of scattering diagrams \cite{GHKK}. The proof there still works in the skew-symmetrizable case as is pointed out by Qin in \cite{qin17}*{Proposition 5.1.4}. We will give an elementary proof of above result based on the sign-coherence of $C$-matrices at the end of this subsection. 
\end{remark}

The following corollary shares the similar spirit with \cite{mur}*{Theorem 9, Theorem 17} and \cite{GARVER201876}*{Theorem 3.3}, which  concern the existence of green-to-red sequences (or its special case maximal green sequences) of certain submatrices of the exchange matrices of a cluster algebra. The following result indicates that many skew-symmetrizable matrices admit green-to-red sequences.

\begin{corollary}\label{corgreen-to-red}
Let  $\mathcal A=\mathcal A(B, t_0)$ be a cluster algebra with initial exchange matrix $B$ at $t_0$ and  $\Sigma_t$ a seed of $\mathcal A$. Let $W$ be the subset of $[{\bf x}_t]$ such that
$x_{i;t}\in W$ if and only if the $i$th column of $C_t^{B;t_0}$ is a non-positive vector. Then the $W\times W$-indexed submatrix $B_t^\dag$ of $B_t$ admits a green-to-red sequence.
\end{corollary}

\begin{proof}
Put $U:=[{\bf x}_t]\setminus W$ and let $\Sigma_s$ be a seed of $\mathcal A$ such that $[{\bf x}_s]=B_U[{\bf x}_{t_0}]$. Put $V:=[{\bf x}_s]\setminus U$ and let $B_s^\dag$ be the $V\times V$-indexed submatrix of $B_s$, and $\mathcal A(B_s^\dag,s)$ be a cluster algebra with initial exchange matrix $B_s^\dag$ at $s$.

By Theorem \ref{thmquiver}, $\overrightarrow{\Gamma}_U(B,t_0)$ is isomorphic to the exchange quiver $\overrightarrow{\Gamma}(B_s^\dag,s)$. It follows from the definition of $W$ that the seed $\Sigma_t$ corresponds to a sink in $\overrightarrow{\Gamma}_U(B,t_0)\cong \overrightarrow{\Gamma}(B_s^\dag,s)$. Then by Remark \ref{rmkquiver} (iii), we know that the initial exchange matrix $B_s^\dag$ of $\mathcal A(B_s^\dag,s)$ admits a green-to-red sequence.

 Since both the seeds $\Sigma_s$ and $\Sigma_t$ correspond to vertices of $\overrightarrow{\Gamma}_U(B,t_0)\cong \overrightarrow{\Gamma}(B_s^\dag,s)$, we know that $B_t^\dag$ (up to a permutation) is an exchange matrix of $\mathcal A(B_s^\dag,s)$. Thus by Proposition \ref{promu}, it admits a green-to-red sequence.
\end{proof}


Now we give an elementary proof of Proposition \ref{promu} based on the sign-coherence of $C$-matrices.

\begin{proof}[Proof of Proposition \ref{promu}]
It suffices to prove that if $B$ has a green-to-red sequence, then so does  $$B^\prime:=\mu_k(B)$$ for any $k=1,\cdots,n$. Recall that $B$ is the exchange matrix at the vertex $t_0\in\mathbb T_n$.

Let $\overset{\leftarrow}{\mu}$ be a green-to-red sequence of $B$ and denote by $t$ the vertex in $\mathbb T_n$ connected with $t_0$ by the mutation sequence $\overset{\leftarrow}{\mu}$. By the definition of green-to-red sequence, we know that the $C$-matrix $C_{t}^{B;t_0}$ is a non-positive matrix. Then by \cite{qin}*{Section 2.3}, we get that both the $C$-matrix $C_{t}^{B;t_0}$ and the $G$-matrix $G_t^{B;t_0}$ are actually equal to the negative identity matrix $-I_n$ up to a permutation $\sigma$ on the column vectors of $-I_n$. For this reason, we denote by $$C_{t}^{B;t_0}=-P_\sigma=G_t^{B;t_0}.$$ Clearly, $P_\sigma\cdot P_\sigma^{\rm T}=I_n$ and thus $P_\sigma^{-1}=P_\sigma^{\rm T}$.

Now we show that $\mu_{k^\prime}\circ\overset{\leftarrow}{\mu}\circ\mu_k$ is a green-to-red sequence of $B^\prime=\mu_k(B)$, where $k^\prime=\sigma^{-1}(k)$. Let $t_0^\prime$ and $t^\prime$ be the two vertices in $\mathbb T_n$ given as follows:
$$\begin{xy}(0,0)*+{t_0}="A",(15,0)*+{t_0^\prime}="B",\ar@{-}^{k}"A";"B" \end{xy}\;\;\;\text{and}\;\;\;\begin{xy}(0,0)*+{t}="A",(15,0)*+{t'}="B",\ar@{-}^{k^\prime}"A";"B" \end{xy}$$
Thus we can go from $t_0^\prime$ to $t^\prime$ in $\mathbb T_n$ along the mutation sequence  $\mu_{k^\prime}\circ\overset{\leftarrow}{\mu}\circ\mu_k$. We want to show that the $C$-matrix $C_{t^\prime}^{B^\prime;t_0^\prime}$ is a non-positive matrix.

Nakanishi-Zelevinsky give two formulas in \cite{nz}*{Equalities (1.16),  (1.15)} to describe how the $C$-matrices change under the initial seed mutation and the final seed mutation, based on the sign-coherence of $C$-matrices. By their formulas, we have
\begin{eqnarray}
\label{eqn:cinitial} 
C^{B^\prime; t_0^\prime}_t&\xlongequal{t\; {\rm fixed}}&\big(J_k+[-\varepsilon_k\cdot B^{k\bullet}]_+\big)\cdot C^{B; t_0}_t,\\
\label{eqn:cfinal}
C^{B^\prime; t_0^\prime}_{t'}&\xlongequal{t_0^\prime\; {\rm fixed}}& C^{B^\prime; t_0^\prime}_t\cdot \big(J_{k^\prime}+[\varepsilon_{k^\prime}\cdot B_t^{k^\prime\bullet}]_+\big)
,
\end{eqnarray}
where 
\begin{itemize}
\item[(a)]$J_{\ell}$ denotes the matrix obtained from $I_n$ by replacing its $(\ell,\ell)$ entry with $-1$;
    \item [(b)] $\varepsilon_{k}$ is the sign of the $k$th column of $C_{t_0}^{-B_t;t}$, which is well defined, thanks to the sign-coherence; 
    \item[(c)]  $\varepsilon_{k^\prime}$ is the sign of the $k^\prime$th column of $C_t^{B^\prime;t_0^\prime}$, which is well defined, thanks to the sign-coherence; 
    \item[(d)] $A^{\ell\bullet}$ denotes the matrix obtained from $A=(a_{ij})$ by setting all entries outside its $\ell$-row to zeros;
    \item[(e)] $[A]_+:=([a_{ij}]_+)$ for a matrix $A=(a_{ij})$.
\end{itemize}

By \cite{nz}*{(1.12)} and since $C_{t}^{B;t_0}=-P_\sigma$, we have 
$$C_{t_0}^{-B_t;t}=(C_{t}^{B;t_0})^{-1}=(-P_{\sigma})^{-1}=-P_{\sigma}^T.$$
Then by the definition of $\varepsilon_{k}$, we have  $\varepsilon_{k}=-1$. Then the equality (\ref{eqn:cinitial}) becomes the following equality.
$$C^{B^\prime; t_0^\prime}_t=\big(J_k+[B^{k\bullet}]_+\big)\cdot (-P_\sigma)=-\big(J_k+[B^{k\bullet}]_+\big)\cdot P_\sigma.$$
We can see that the $k^\prime$th column (i.e., $\sigma^{-1}(k)$th column) of $C_t^{B^\prime;t_0^\prime}$ is just the $k$th column of $-(J_k+[B^{k\bullet}]_+\big)$, which coincides with the $k$th column of $I_n$. Then by the definition of  $\varepsilon_{k^\prime}$, we have  $\varepsilon_{k^\prime}=1$. Then the equality (\ref{eqn:cfinal}) becomes the following equality.
\begin{eqnarray}
C^{B^\prime; t_0^\prime}_{t'}&=&C^{B^\prime; t_0^\prime}_t\cdot \big(J_{k^\prime}+[ B_t^{k^\prime\bullet}]_+\big)\nonumber\\
&=&-\big(J_k+[B^{k\bullet}]_+\big)\cdot P_\sigma \cdot \big(J_{k^\prime}+[ B_t^{k^\prime\bullet}]_+\big).\nonumber
\end{eqnarray}

By \cite{nz}*{(2.8)}, we have $G_t^{B;t_0}B_t=BC_t^{B;t_0}$. Then by  $C_{t}^{B;t_0}=-P_\sigma=G_t^{B;t_0}$, we get
 $P_\sigma B_t=BP_\sigma$. Thus we have $P_\sigma [B_t^{k^\prime\bullet}]_+=[B^{k\bullet}]_+P_\sigma$ and 
 $$P_\sigma\cdot \big(J_{k^\prime}+[ B_t^{k^\prime\bullet}]_+\big)=\big(J_k+[B^{k\bullet}]_+\big)\cdot P_\sigma.$$
 So we have
\begin{eqnarray}
C^{B^\prime; t_0^\prime}_{t'}
&=&-\big(J_k+[B^{k\bullet}]_+\big)\cdot P_\sigma \cdot \big(J_{k^\prime}+[ B_t^{k^\prime\bullet}]_+\big)\nonumber\\
&=&-\big(J_k+[B^{k\bullet}]_+\big)\cdot \big(J_k+[B^{k\bullet}]_+\big)\cdot P_\sigma.\nonumber
\end{eqnarray}

Notice that $\big(J_k+[B^{k\bullet}]_+\big)^2=I_n$. So we get
$C^{B^\prime; t_0^\prime}_{t'}=-P_\sigma$, which is a non-positive matrix. This implies that $\mu_{k^\prime}\circ\overset{\leftarrow}{\mu}\circ\mu_k$ is a green-to-red sequence of $B^\prime=\mu_k(B)$.
\end{proof}

\subsection{Application to $Y$-patterns}
Let $\Sigma^Y=\{t\mapsto \Sigma_t^Y\}_{t\in\mathbb T_n}$ be a $Y$-pattern over a universal semifield ${\mathbb Q}_\mathrm{sf}(u_1,\cdots,u_n)$ such that  $${\bf y}_{t_0}=(y_{1;t_0},\cdots,y_{n;t_0})=(u_1,\cdots,u_n),$$ where $\Sigma_{t_0}^Y=({\bf y}_{t_0}, B_{t_0})$ is the $Y$-seed at the rooted vertex $t_0\in\mathbb T_n$. The coefficient variables appearing in $\Sigma^Y=\{t\mapsto \Sigma_t^Y\}_{t\in\mathbb T_n}$ are called {\em $y$-variables}. 

Notice that each $y$-variable $y_{k;t}$ belongs to ${\mathbb Q}_\mathrm{sf}(y_{1;t_0},\cdots,y_{n;t_0})$. Thus we can consider the tropical evaluation of $y_{k;t}$ in $\mathbb P=\trop(y_{1;t_0},\cdots,y_{n;t_0})$ and it takes the form
 $$y_{k;t}|_{\mathbb P}=y_{1;t_0}^{c_{1k}}\cdots y_{n;t_0}^{c_{nk}},$$
where $\mathbf{ c}_{k;t}^{B_{t_0};t_0}=(c_{1k},\cdots,c_{nk})^\mathrm{T}$ is the $k$th column vector of the $C$-matrix $C_t^{B_{t_0};t_0}$ \cite{fziv}.
Thanks to sign-coherence, we have that either $\mathbf{ c}_{k;t}^{B_{t_0};t_0}\in\mathbb N^n$ or $\mathbf{ c}_{k;t}^{B_{t_0};t_0}\in(-\mathbb N)^n$.

\begin{definition}[Positive and negative $y$-variables]
Keep the above notations. The $y$-variable $y_{k;t}$ is called \emph{positive} if $\mathbf{c}_{k;t}^{B_{t_0};t_0}\in\mathbb N^n$ and it is called \emph{negative} if $\mathbf{ c}_{k;t}^{B_{t_0};t_0}\in(-\mathbb N)^n$.
\end{definition}

\begin{lemma}[\cite{ckq}] \label{lemckq}
Let $\Sigma^Y=\{t\mapsto \Sigma_t^Y\}_{t\in\mathbb T_n}$ be a $Y$-pattern over a universal semifield ${\mathbb Q}_\mathrm{sf}(u_1,\cdots,u_n)$ such that $${\bf y}_{t_0}=(y_{1;t_0},\cdots,y_{n;t_0})=(u_1,\cdots,u_n),$$ where $\Sigma_{t_0}^Y=({\bf y}_{t_0}, B_{t_0})$ is the $Y$-seed at the rooted vertex $t_0\in\mathbb T_n$. Suppose that $\Sigma=\{t\mapsto \Sigma_t\}_{t\in\mathbb T_n}$ is a cluster pattern over a semifield $\mathbb P$ such that it has the same exchange matrix with $\Sigma^Y$ at the rooted vertex $t_0\in\mathbb T_n$. Then $\alpha: y_{k;t}\mapsto x_{k;t}$ induces a well-defined map from the $y$-variables of $\Sigma^Y$ to the cluster variables of $\Sigma$.
\end{lemma}

\begin{proof}
In the finite type case, this result can be found in \cite{Sherman-Bennett-2019}*{Corollary 1.2} and the general case can be found in \cite{ckq}*{Theorem 7.6}. For convenience of the reader, we include a proof here, which is slightly different from the one in  \cite{ckq}. However, the idea for the proof essentially comes from \cite{ckq}.

In order to show that the map $\alpha: y_{k;t}\mapsto x_{k;t}$ is well defined, it suffices to show that 
if $y_{k;t}=y_{j;t^\prime}$ for some $j$ and $t^\prime$, then $x_{k;t}=x_{j;t^\prime}$. Since we can choose any vertex of $\mathbb T_n$ as a rooted vertex. It suffices for us to show the case that if $y_{k;t}=y_{j;t_0}$, then $x_{k;t}=x_{j;t_0}$.

By Proposition \ref{proindepen}, we can assume without loss of generality that the cluster pattern $\Sigma=\{t\mapsto \Sigma_t\}_{t\in\mathbb T_n}$ has principal coefficients at the vertex $t_0$ and we write the coefficients tuple in the seed $\Sigma_t$ as ${\bf y}_t^{\rm pr}=(y_{1;t}^{\rm pr},\cdots,y_{n;t}^{\rm pr})$ to distinguish the coefficients tuple in the $Y$-seed $({\bf y}_t, B_t)$ of the given $Y$-pattern.

We denote by $\hat y_{k;t}^{\rm pr}:=y_{k;t}^{\rm pr}\cdot \prod\limits_{i=1}^nx_{i;t}^{b_{ik;t}}$ and $\hat {\bf y}_t^{\rm pr}=(\hat y_{1;t}^{\rm pr},\cdots,\hat y_{n;t}^{\rm pr})$. By \cite{fziv}*{Proposition 3.9}, we know that $\{t\mapsto (\hat {\bf y}_t^{\rm pr},B_t)\}_{t\in\mathbb T_n}$ forms a $Y$-pattern. Now we have three $Y$-patterns: $\Sigma^Y=\{t\mapsto \Sigma_t^Y\}_{t\in\mathbb T_n}$, $\{t\mapsto (\hat {\bf y}_t^{\rm pr},B_t)\}_{t\in\mathbb T_n}$ and the underlying $Y$-pattern $\{t\mapsto ({\bf y}_t^{\rm pr},B_t)\}_{t\in\mathbb T_n}$ of the cluster pattern $\Sigma=\{t\mapsto \Sigma_t\}_{t\in\mathbb T_n}$. The three $Y$-patterns have the same initial exchange matrix $B_{t_0}$.
Since the $Y$-pattern $\Sigma^Y=\{t\mapsto \Sigma_t^Y\}_{t\in\mathbb T_n}$ is over a universal semifield ${\mathbb Q}_\mathrm{sf}(u_1,\cdots,u_n)$ with $${\bf y}_{t_0}=(y_{1;t_0},\cdots,y_{n;t_0})=(u_1,\cdots,u_n),$$
we know that $y_{k;t}=y_{j;t_0}$ implies that $\hat y_{k;t}^{\rm pr}=\hat y_{j;t_0}^{\rm pr}$ and $y_{k;t}^{\rm pr}=y_{j;t_0}^{\rm pr}$. Notice that $\hat y_{k;t}^{\rm pr}=y_{k;t}^{\rm pr}\cdot \prod\limits_{i=1}^nx_{i;t}^{b_{ik;t}}$ and $\hat y_{j;t_0}^{\rm pr}=y_{j;t_0}^{\rm pr}\cdot \prod\limits_{i=1}^nx_{i;t_0}^{b_{ij;t_0}}$. Hence, we have $\prod\limits_{i=1}^nx_{i;t}^{b_{ik;t}}=\prod\limits_{i=1}^nx_{i;t_0}^{b_{ij;t_0}}$ and thus 

\begin{eqnarray}\label{eqn:ufd}
\prod\limits_{i=1}^nx_{i;t}^{[b_{ik;t}]_+}\cdot \prod\limits_{i=1}^nx_{i;t_0}^{[-b_{ij;t_0}]_+}=\prod\limits_{i=1}^nx_{i;t_0}^{[b_{ij;t_0}]_+}\cdot \prod\limits_{i=1}^nx_{i;t}^{[-b_{ik;t}]_+}.
\end{eqnarray}

Notice that the above equality is an equality in the upper cluster algebra $\mathcal U$ of the cluster pattern $\Sigma=\{t\mapsto \Sigma_t\}_{t\in\mathbb T_n}$ with principal coefficients at the vertex $t_0$. By \cite{GHK15}*{Corollary 4.7}, we know that $\mathcal U$ is a unique factorization domain (this result is also proved in \cite{ckq} with a very different method). By \cite{GLS13}*{Theorem 3.1} (proved for cluster algebras, but still true for upper cluster algebras), we know that the cluster variables in $\mathcal U$ are irreducible elements. By \cite{GLS13}*{Corollary 2.3 (ii)} (proved for cluster algebras, but still true for upper cluster algebras), we know that two cluster variables are associate in $\mathcal U$ if and only if they are equal. Based on the above reasons and by the equality (\ref{eqn:ufd}), we get
$$\prod\limits_{i=1}^nx_{i;t}^{[b_{ik;t}]_+}=\prod\limits_{i=1}^nx_{i;t_0}^{[b_{ij;t_0}]_+}\;\;\;\text{ and }\;\;\; \prod\limits_{i=1}^nx_{i;t}^{[-b_{ik;t}]_+}=\prod\limits_{i=1}^nx_{i;t_0}^{[-b_{ij;t_0}]_+}.$$
Then by $y_{k;t}^{\rm pr}=y_{j;t_0}^{\rm pr}$, we get that the exchange binomial of the seed $\Sigma_t$ in direction $k$ and the exchange binomial of the seed $\Sigma_{t_0}$ in direction $j$ are equal, i.e.,

\begin{eqnarray}\label{eqn:bino}
\frac{y_{k;t}^{\rm pr}\prod\limits_{i=1}^nx_{i;t}^{[b_{ik;t}]_+}+\prod\limits_{i=1}^nx_{i;t}^{[-b_{ik;t}]_+}}{1\oplus y_{k;t}^{\rm pr}}=\frac{y_{j;t_0}^{\rm pr}\prod\limits_{i=1}^nx_{i;t_0}^{[b_{ij;t_0}]_+}+\prod\limits_{i=1}^nx_{i;t_0}^{[-b_{ij;t_0}]_+}}{1\oplus y_{j;t_0}^{\rm pr}}.
\end{eqnarray}
Hence, we have $x_{k;t}x_{k;t}^\prime=x_{j;t_0}x_{j;t_0}^\prime$, where $x_{k;t}^\prime$ (resp., $x_{j;t_0}^\prime$) is the new cluster variable in the seed $\mu_k(\Sigma_t)$ (resp., $\mu_j(\Sigma_{t_0})$). Since $\mathcal U$ is a unique factorization domain and the cluster variables are irreducible in $\mathcal U$ and by \cite{GLS13}*{Corollary 2.3 (ii)} again, we know that $x_{k;t}x_{k;t}^\prime=x_{j;t_0}x_{j;t_0}^\prime$ implies either $(x_{k;t},x_{k;t}^\prime)=(x_{j;t_0},x_{j;t_0}^\prime)$ or $(x_{k;t},x_{k;t}^\prime)=(x_{j;t_0}^\prime,x_{j;t_0})$.

In the remaining part, we will rule out the possibility of $(x_{k;t},x_{k;t}^\prime)=(x_{j;t_0}^\prime,x_{j;t_0})$ under the condition $y_{k;t}^{\rm pr}=y_{j;t_0}^{\rm pr}$.

Let $C_t=(c_{ij})$ and $G_t=(g_{ij})$ be the $C$-matrix and $G$-matrix of the seed $\Sigma_t$ with respect to the initial seed $\Sigma_{t_0}$. By Theorem \ref{thmcg}, we have
\begin{eqnarray}\label{eqn:cgid}
(S^{-1}G_t^{\rm T}S)\cdot C_t=I_n,
\end{eqnarray}
where $S=\mathrm{diag}(s_1,\dots,s_n)$ is a skew-symmetrizer of $B_{t_0}$.
By $y_{k;t}^{\rm pr}=y_{j;t_0}^{\rm pr}$, we know that the $k$th column of $C_t$ is ${\bf e}_j$, where ${\bf e}_j$ is the $j$th column of the identity matrix $I_n$. By comparing the $k$th column of the two sides of the equality (\ref{eqn:cgid}), we get that the $j$th column of $(S^{-1}G_t^{\rm T}S)$ equals the $k$th column of $I_n$. In particular, the $(k,j)$-entry $s_k^{-1}\cdot g_{jk}\cdot s_j$ of $(S^{-1}G_t^{\rm T}S)$ equals $1$, and thus $$g_{jk}=\frac{s_k}{s_j}>0.$$

Recall that we have either $(x_{k;t},x_{k;t}^\prime)=(x_{j;t_0},x_{j;t_0}^\prime)$ or $(x_{k;t},x_{k;t}^\prime)=(x_{j;t_0}^\prime,x_{j;t_0})$. We claim that $x_{k;t}$ can not be $x_{j;t_0}^\prime$. Otherwise, the $j$th component of the $g$-vector of $x_{k;t}=x_{j;t_0}^\prime$ is $-1$, because $x_{j;t_0}^\prime$ is the new cluster variable in $\mu_j(\Sigma_{t_0})$. This contradicts $g_{jk}=\frac{s_k}{s_j}>0$. So the possibility of $(x_{k;t},x_{k;t}^\prime)=(x_{j;t_0}^\prime,x_{j;t_0})$ is ruled out. Thus we must have $(x_{k;t},x_{k;t}^\prime)=(x_{j;t_0},x_{j;t_0}^\prime)$. In particular, $x_{k;t}=x_{j;t_0}$. So the map $\alpha: y_{k;t}\mapsto x_{k;t}$ is well defined.
\end{proof}

\begin{theorem}\label{thmpoisson}
Let $\Sigma^Y=\{t\mapsto \Sigma_t^Y\}_{t\in\mathbb T_n}$ be a $Y$-pattern over a universal semifield ${\mathbb Q}_\mathrm{sf}(u_1,\cdots,u_n)$ such that $${\bf y}_{t_0}=(y_{1;t_0},\cdots,y_{n;t_0})=(u_1,\cdots,u_n),$$ where $\Sigma_{t_0}^Y=({\bf y}_{t_0}, B_{t_0})$ is the $Y$-seed at the rooted vertex $t_0\in\mathbb T_n$. Then each $Y$-seed $(\mathbf{ y}_t, B_t)$ of $\Sigma^Y$ up to a $Y$-seed equivalence is uniquely determined by the set of negative $y$-variables in $\mathbf{ y}_t$.
\end{theorem}

\begin{proof}

Let $V=\{y_{k_1;t},\cdots,y_{k_\ell;t}\}$ be the set of negative $y$-variables in $\mathbf{ y}_t$. Then the $i$-column vector of the $C$-matrix $C_t^{B_{t_0};t_0}$ belongs to $\mathbb N^n$ for any $i\neq k_1,\cdots,k_\ell$.

Let $\Sigma=\{t\mapsto \Sigma_t\}_{t\in\mathbb T_n}$ be a cluster pattern such that its underlying $Y$-pattern is exactly the given $Y$-pattern $\Sigma^Y$. Clearly, such cluster patterns always exist. Let $\alpha: y_{k;t}\mapsto x_{k;t}$ be the well-defined map in Lemma \ref{lemckq}. We set
$$U=\alpha(V)=\{\alpha(y_{k_1;t}),\cdots,\alpha(y_{k_\ell;t})\}=\{x_{k_1;t},\cdots,x_{k_\ell;t}\},$$ which is a subset of $[\mathbf{ x}_t]$.
Because the $i$th column vector of $C_t^{B_{t_0};t_0}$ is non-negative for any $i\neq k_1,\cdots,k_\ell$, we know that $[\mathbf{ x}_t]$ is the Bongartz completion of $U$ with respect to $t_0$, i.e., $$[{\bf x}_t]=B_U[{\bf x}_{t_0}]=B_{\alpha(V)}[{\bf x}_{t_0}].$$
Hence, $[{\bf x}_t]$ is uniquely determined by $V$, because the initial seed $\Sigma_{t_0}=({\bf x}_{t_0}, {\bf y}_{t_0}, B_{t_0})$ is given.
 Then by Theorem \ref{non-labeled-cluster-thm}, we get that the seed $\Sigma_t=(\mathbf{ x}_t,\mathbf{ y}_t,B_t)$ is uniquely determined by $V$ up to a seed equivalence. This in particular implies that the $Y$-seed $\Sigma_t^Y=(\mathbf{ y}_t, B_t)$ is uniquely determined by $V$ up to  a $Y$-seed equivalence.
\end{proof}

Notice that the existence of Bongartz completion in above proof follows from the definition of Bongartz completion. Actually, the above proof mainly depends on the uniqueness of Bongartz completion. Similarly, if one uses the definition of Bongartz co-completion and the uniqueness of Bongartz co-completion in Corollary \ref{corunique}(iii), one can show the following dual result.

\begin{theorem}
Let $\Sigma^Y=\{t\mapsto \Sigma_t^Y\}_{t\in\mathbb T_n}$ be a $Y$-pattern over a universal semifield ${\mathbb Q}_\mathrm{sf}(u_1,\cdots,u_n)$ such that  $${\bf y}_{t_0}=(y_{1;t_0},\cdots,y_{n;t_0})=(u_1,\cdots,u_n),$$ where $\Sigma_{t_0}^Y=({\bf y}_{t_0}, B_{t_0})$ is the $Y$-seed at the rooted vertex $t_0\in\mathbb T_n$.  Then each $Y$-seed $(\mathbf{ y}_t, B_t)$ of $\Sigma^Y$ up to a $Y$-seed equivalence is uniquely determined by the set of positive $y$-variables in $\mathbf{ y}_t$.
\end{theorem}

We remark that the results in this subsection is actually inspired by Changjian Fu's online talk \cite{fu2021} on January 15, 2021.

\section*{Acknowledgement}
The authors are grateful to Changjian Fu, Bernhard Keller and Dylan Rupel for their helpful comments and advices. P. Cao would like to thank Hualin Huang, Zengqiang Lin and Zhankui Xiao for their hospitality during his visiting Huaqiao University (March 16--April 9, 2021), where some of the results of this article were carried out.

P. Cao is supported by the ERC (Grant No. 669655), the NSF of China (Grant No. 12071422), and  the
Guangdong Basic and Applied Basic Research Foundation (Grant No. 2021A1515012035). Y. Gyoda is supported by JSPS KAKENHI Grant Number JP20J12675. T. Yurikusa is supported by JSPS KAKENHI Grant Numbers JP20J00410, JP21K13761.
\bibliography{myrefs}
\end{document}